\title{Flows on measurable spaces}
\author{L\'aszl\'o Lov\'asz\footnote{Alfr\'ed R\'enyi Institute of Mathematics. Research supported by ERC Synergy Grant
No.~810115.}}

\documentclass[fleqn]{article}

\usepackage{amsmath,amssymb,amsthm,graphicx,bbm,bm,delarray}
\usepackage{srcltx}
\usepackage[hypertex]{hyperref}

\usepackage[margin=30pt,font=small,labelfont=bf]{caption}

\newtheorem{theorem}{Theorem}[section]
\newtheorem{prop}[theorem]{Proposition}
\newtheorem{lemma}[theorem]{Lemma}
\newtheorem{claim}{Claim}
\newtheorem{corollary}[theorem]{Corollary}
\newtheorem{suppl}[theorem]{Supplement}

\theoremstyle{definition}
\newtheorem{remark}[theorem]{Remark}
\newtheorem{example}[theorem]{Example}
\newtheorem{conjecture}{Conjecture}

\newenvironment{proof*}[1]{\medskip\noindent{\it Proof of #1.}}{\hfill$\square$\medskip}

\begin{document}

\addtolength{\baselineskip}{3pt}

\def\AA{\mathcal{A}}\def\BB{\mathcal{B}}\def\CC{\mathcal{C}}
\def\DD{\mathcal{D}}\def\EE{\mathcal{E}}\def\FF{\mathcal{F}}
\def\GG{\mathcal{G}}\def\HH{\mathcal{H}}\def\II{\mathcal{I}}
\def\JJ{\mathcal{J}}\def\KK{\mathcal{K}}\def\LL{\mathcal{L}}
\def\MM{\mathcal{M}}\def\NN{\mathcal{N}}\def\OO{\mathcal{O}}
\def\PP{\mathcal{P}}\def\QQ{\mathcal{Q}}\def\RR{\mathcal{R}}
\def\SS{\mathcal{S}}\def\TT{\mathcal{T}}\def\UU{\mathcal{U}}
\def\VV{\mathcal{V}}\def\WW{\mathcal{W}}\def\XX{\mathcal{X}}
\def\YY{\mathcal{Y}}\def\ZZ{\mathcal{Z}}

\def\Ab{\mathbf{A}}\def\Bb{\mathbf{B}}\def\Cb{\mathbf{C}}
\def\Db{\mathbf{D}}\def\Eb{\mathbf{E}}\def\Fb{\mathbf{F}}
\def\Gb{\mathbf{G}}\def\Hb{\mathbf{H}}\def\Ib{\mathbf{I}}
\def\Jb{\mathbf{J}}\def\Kb{\mathbf{K}}\def\Lb{\mathbf{L}}
\def\Mb{\mathbf{M}}\def\Nb{\mathbf{N}}\def\Ob{\mathbf{O}}
\def\Pb{\mathbf{P}}\def\Qb{\mathbf{Q}}\def\Rb{\mathbf{R}}
\def\Sb{\mathbf{S}}\def\Tb{\mathbf{T}}\def\Ub{\mathbf{U}}
\def\Vb{\mathbf{V}}\def\Wb{\mathbf{W}}\def\Xb{\mathbf{X}}
\def\Yb{\mathbf{Y}}\def\Zb{\mathbf{Z}}

\def\ab{\mathbf{a}}\def\bb{\mathbf{b}}\def\cb{\mathbf{c}}
\def\db{\mathbf{d}}\def\eb{\mathbf{e}}\def\fb{\mathbf{f}}
\def\gb{\mathbf{g}}\def\hb{\mathbf{h}}\def\ib{\mathbf{i}}
\def\jb{\mathbf{j}}\def\kb{\mathbf{k}}\def\lb{\mathbf{l}}
\def\mb{\mathbf{m}}\def\nb{\mathbf{n}}\def\ob{\mathbf{o}}
\def\pb{\mathbf{p}}\def\qb{\mathbf{q}}\def\rb{\mathbf{r}}
\def\sb{\mathbf{s}}\def\tb{\mathbf{t}}\def\ub{\mathbf{u}}
\def\vb{\mathbf{v}}\def\wb{\mathbf{w}}\def\xb{\mathbf{x}}
\def\yb{\mathbf{y}}\def\zb{\mathbf{z}}

\def\Abb{\mathbb{A}}\def\Bbb{\mathbb{B}}\def\Cbb{\mathbb{C}}
\def\Dbb{\mathbb{D}}\def\Ebb{\mathbb{E}}\def\Fbb{\mathbb{F}}
\def\Gbb{\mathbb{G}}\def\Hbb{\mathbb{H}}\def\Ibb{\mathbb{I}}
\def\Jbb{\mathbb{J}}\def\Kbb{\mathbb{K}}\def\Lbb{\mathbb{L}}
\def\Mbb{\mathbb{M}}\def\Nbb{\mathbb{N}}\def\Obb{\mathbb{O}}
\def\Pbb{\mathbb{P}}\def\Qbb{\mathbb{Q}}\def\Rbb{\mathbb{R}}
\def\Sbb{\mathbb{S}}\def\Tbb{\mathbb{T}}\def\Ubb{\mathbb{U}}
\def\Vbb{\mathbb{V}}\def\Wbb{\mathbb{W}}\def\Xbb{\mathbb{X}}
\def\Ybb{\mathbb{Y}}\def\Zbb{\mathbb{Z}}

\def\Af{\mathfrak{A}}\def\Bf{\mathfrak{B}}\def\Cf{\mathfrak{C}}
\def\Df{\mathfrak{D}}\def\Ef{\mathfrak{E}}\def\Ff{\mathfrak{F}}
\def\Gf{\mathfrak{G}}\def\Hf{\mathfrak{H}}\def\If{\mathfrak{I}}
\def\Jf{\mathfrak{J}}\def\Kf{\mathfrak{K}}\def\Lf{\mathfrak{L}}
\def\Mf{\mathfrak{M}}\def\Nf{\mathfrak{N}}\def\Of{\mathfrak{O}}
\def\Pf{\mathfrak{P}}\def\Qf{\mathfrak{Q}}\def\Rf{\mathfrak{R}}
\def\Sf{\mathfrak{S}}\def\Tf{\mathfrak{T}}\def\Uf{\mathfrak{U}}
\def\Vf{\mathfrak{V}}\def\Wf{\mathfrak{W}}\def\Xf{\mathfrak{X}}
\def\Yf{\mathfrak{Y}}\def\Zf{\mathfrak{Z}}

\def\afr{\mathfrak{a}}\def\bfr{\mathfrak{b}}\def\cfr{\mathfrak{c}}
\def\dfr{\mathfrak{d}}\def\efr{\mathfrak{e}}\def\ffr{\mathfrak{f}}
\def\gfr{\mathfrak{g}}\def\hfr{\mathfrak{h}}\def\ifr{\mathfrak{i}}
\def\jfr{\mathfrak{j}}\def\kfr{\mathfrak{k}}\def\lfr{\mathfrak{l}}
\def\mfr{\mathfrak{m}}\def\nfr{\mathfrak{n}}\def\ofr{\mathfrak{o}}
\def\pfr{\mathfrak{p}}\def\qfr{\mathfrak{q}}\def\rfr{\mathfrak{r}}
\def\sfr{\mathfrak{s}}\def\tfr{\mathfrak{t}}\def\ufr{\mathfrak{u}}
\def\vfr{\mathfrak{v}}\def\Wfr{\mathfrak{w}}\def\xfr{\mathfrak{x}}
\def\yfr{\mathfrak{y}}\def\zfr{\mathfrak{z}}

\def\alphab{{\boldsymbol\alpha}}\def\betab{{\boldsymbol\beta}}
\def\gammab{{\boldsymbol\gamma}}\def\deltab{{\boldsymbol\delta}}
\def\etab{{\boldsymbol\eta}}\def\zetab{{\boldsymbol\zeta}}
\def\kappab{{\boldsymbol\kappa}}
\def\lambdab{{\boldsymbol\lambda}}\def\mub{{\boldsymbol\mu}}
\def\nub{{\boldsymbol\nu}}\def\pib{{\boldsymbol\pi}}
\def\rhob{{\boldsymbol\rho}}\def\sigmab{{\boldsymbol\sigma}}
\def\taub{{\boldsymbol\tau}}\def\epsb{{\boldsymbol\varepsilon}}
\def\phib{{\boldsymbol\varphi}}\def\psib{{\boldsymbol\psi}}
\def\xib{{\boldsymbol\xi}}\def\omegab{{\boldsymbol\omega}}

\def\RSA{{\it Random Struc.\ Alg.} }
\def\CCA{{\it Combinatorica} }
\def\JCTB{{\it J.~Combin.\ Theory B} }
\def\JCTA{{\it J.~Combin.\ Theory A} }
\def\CPC{{\it Combin.\ Prob.\ Comput.} }
\def\EJC{{\it Europ.\ J.~Combin.} }
\def\ELJC{{\it Electr.\ J.~Combin.} }
\def\GC{{\it Graphs and Combin.} }
\def\JGT{{\it J.~Graph Theory} }
\def\ADV{{\it Advances in Math.} }
\def\ADVA{{\it Advances in Applied Math.} }
\def\AMH{{\it Acta Math. Hung.} }
\def\GAFA{{\it Geom.\ Func.\ Anal.} }
\def\STOC#1 {{\it Proc.\ #1$^\text{th}$ ACM Symp.\
on Theory of Comput.} }
\def\FOCS#1 {{\it Proc.\ #1$^\text{th}$ Ann.\ IEEE
Symp.\ on Found.\ Comp.\ Science} }
\def\SODA#1 {{\it Proc.\ #1$^\text{th}$ Ann.\
ACM-SIAM Symp.\ on Discrete Algorithms} }
\def\DCG{{\it Discr.\ Comput.\ Geom.} }
\def\DM{{\it Discr.\ Math.} }
\def\DAM{{\it Discr.\ Applied Math.} }
\def\SJC{{\it SIAM J.~Comput.} }
\def\SDM{{\it SIAM J.~Discr.\ Math.} }
\def\TIT{{\it IEEE Trans.\ Inform.\ Theory} }
\def\LAA{{\it Linear Algebra Appl.} }

\def\sqprod{\mathbin{\square}}

\def\ybb{\mathbbm{y}}
\def\one{{\mathbbm1}}
\def\two{{\mathbbm2}}
\def\R{\Rbb}\def\Q{\Qbb}\def\Z{\Zbb}\def\N{\Nbb}\def\C{\Cbb}
\def\wh{\widehat}
\def\wt{\widetilde}

\def\eps{\varepsilon}
\def\ca{\Mf}
\def\sgn{{\rm sgn}}
\def\dd{{\sf d}}
\def\Rv{\overleftarrow}
\def\Pr{{\sf P}}
\def\E{{\sf E}}
\def\T{^{\sf T}}
\def\proofend{\hfill$\square$}
\def\id{\hbox{\rm id}}
\def\conv{\hbox{\rm conv}}
\def\lin{\hbox{\rm lin}}
\def\conv{\hbox{\rm conv}}
\def\Dim{\hbox{\rm Dim}}
\def\const{\hbox{\rm const}}
\def\vol{\text{\rm vol}}
\def\val{\text{\rm val}}
\def\diam{\text{\rm diam}}
\def\corank{\hbox{\rm cork}}
\def\cork{\hbox{\rm cork}}
\def\cro{\text{\rm cr}}
\def\supp{\text{\rm supp}}
\def\grad{\text{\rm grad}}
\def\rk{\hbox{\rm rk}}
\def\srk{\hbox{\rm srk}}
\def\diag{{\rm diag}}
\def\pw{{\sf w}_\text{\rm prod}}
\def\tw{{\sf w}_\text{\rm tree}}
\def\aw{{\sf w}_\text{\rm alg}}
\def\bw{{\sf bw}}
\def\ld{{\sf d}_{\rm loc}}
\def\hd{{\sf d}_{\rm har}}
\def\tv{\text{\rm tv}}
\def\Tr{\text{\rm Tr}}
\def\tr{\text{\rm tr}}
\def\Prob{\hbox{\rm Pr}}
\def\bl{\text{{\rm bl}}}
\def\Inf{\text{\sf Inf}}
\def\Str{\text{\sf Str}}
\def\Rig{\text{\sf Rig}}
\def\Mat{\text{\sf Mat}}
\def\comm{{\sf comm}}
\def\maxcut{{\sf maxcut}}
\def\disc{\text{\sf disc}}
\def\cond{\Phi}
\def\dist{d_{\rm qu}}
\def\dhaus{d_{\rm haus}}
\def\dlp{d_{\rm LP}}
\def\dact{d_{\rm act}}
\def\Rng{\text{\rm Rng}}
\def\Ker{\text{\rm Ker}}

\long\def\ignore#1{}
\def\gdim{{\rm gdim}}
\def\gap{\text{\rm gap}}

\maketitle

\tableofcontents

\newpage

\begin{abstract}
The theory of graph limits is only understood to a somewhat satisfactory degree
in the cases of dense graphs and of bounded degree graphs. There is, however, a
lot of interest in the intermediate cases. It appears that one of the most
important constituents of graph limits in the general case will be Markov
spaces (Markov chains on measurable spaces with a stationary distribution).

This motivates our goal to extend some important theorems from finite graphs to
Markov spaces or, more generally, to measurable spaces. In this paper, we show
that much of flow theory, one of the most important areas in graph theory, can
be extended to measurable spaces. Surprisingly, even the Markov space structure
is not fully needed to get these results: all we need a standard Borel space
with a measure on its square (generalizing the finite node set and the counting
measure on the edge set). Our results may be considered as extensions of flow
theory for directed graphs to the measurable case.
\end{abstract}

\section{Introduction}

The theory graph limits is only understood to a somewhat satisfactory degree in
the case of dense graphs, where the limit objects are {\it graphons}, and (on
the opposite end of the scale) in the case of bounded degree graphs, where the
limit objects are {\it graphings}. There is, however, a lot of work being done
on the intermediate cases. It appears that the most important constituents of
graph limits in the general case will be Markov spaces (Markov chains on
measurable spaces with a stationary distribution). Markov spaces can be
described by a (boolean) sigma-algebra, endowed with a measure on its square,
such that its two marginals are equal.

A finite directed graph $G=(V,E)$ can be thought of as a sigma-algebra $2^V$,
endowed with a measure of $V\times V$, the counting measure of the set of
edges. This motivates our goal to extend some important theorems from finite
graphs to measures on squares of sigma-algebras. In this paper we show that
much of flow theory, one of the most important areas in graph theory, can be
extended to such spaces.

In the finite case, a flow is a function on the edges; we often sum its values
on subsets of edges (e.g.~cuts), which means we are also using the
corresponding measure on subsets. In the case of an infinite point set $J$
(endowed with a sigma-algebra $\AA$), these two notions diverge: we can try to
generalize the notion of a flow either as a function on ordered pairs of
points, or as a measure on the subsets of $J\times J$ measurable with respect
to the sigma-algebra $\AA\times\AA$. While the first notion is perhaps more
natural, flows as measures are easier to define, and we explore this
possibility in this paper. Note that even the definition of the flow condition
``inflow$=$outflow'' in the infinite case needs some additional hypothesis or
stucture: Laczkovich \cite{Lacz} uses an underlying measure on the nodes, while
and Marks and Unger \cite{MU} restrict their attention to finite-degree graphs.
Of course, one can get back and force between measures and functions under
under appropriate circumstances (by integration and Radon-Nikodym
differentiation, respectively), but the measure-theoretic formulation seems to
involve the least number of extra conditions.

In particular, we generalize the Hoffman Circulation Theorem to measurable
spaces. This connects us with the theory of Markov spaces, which can be
described as measurable spaces endowed with a nonnegative normalized
circulation, called the {\it ergodic circulation}. Our main concern will be the
{\it existence} of circulations; in this sense, these studies can be thought of
as preliminaries for the study of Markov spaces or Markov chains, which are
concerned with measurable spaces with a {\it given} ergodic circulation.

Flows between two points, and more generally, between two measures can then be
handled using the results about circulations (by the same reductions as in the
finite case). In particular, we prove an extension of the Max-Flow-Min-Cut
Theorem, and a measure-theoretic generalization of the Multicommodity Flow
Theorem by Iri and Matula--Shahroki.

A few caveats: Graph limit theory has served as the motivation of these
studies, but in this paper we don't study how, for a graph sequence that is
convergent in some well-defined sense, parameters and properties of flows
converge to those of flows on the measurable spaces serving as their limit
objects.

Also, Markov spaces only capture the edge measure of graphons and graphings; to
get a proper generalization, one needs to add a further measure on the nodes,
to get a {\it double measure space}. This node measure is not needed for our
development of measure-theoretic flow theory, but it is clearly needed for
extending other graph-theoretic notions, like expansion or matchings (see e.g.\
\cite{HP}).

Third, our proofs for the existence various (generalized) flows in this paper
are not constructive, because of the use of the Hahn--Banach Theorem. Of
course, in these infinite structures no ``algorithmic'' proof can be given, but
replacing our proofs by iterative constructions modeled on algorithmic proofs
in the finite setting would be desirable.

\section{Preliminaries}

\subsection{Flow theory on finite graphs}

As a motivation of the results in this paper, let us recall some basic results
on finite graphs in this area.

Let $G=(V,E)$ be a finite directed graph and $g:~E\to\R$. The {\it flow
condition} at node $i$ is that the ``inflow'' equals the ``outflow''; formally,
\begin{equation}\label{EQ:FLOW-COND}
\sum_{j:\,ij\in E} g(ij) = \sum_{j:\,ji\in E} g(ji)
\end{equation}
A {\it circulation} on $G$ is a function $f:~E\to\R$ satisfying the flow
condition at every node $i$. Circulations could also be defined by the
condition
\[
\sum_{i\in A, j\in A^c} g(ij) = \sum_{i\in A^c, j\in A} g(ij)
\]
for every $A\subseteq V$ (here $A^c=V\setminus A$ denotes the complement of
$A$). A basic result about the existence of circulations satisfying prescribed
bounds is the following \cite{Hoff}.

\medskip

\noindent{\bf Hoffman's Circulation Theorem.} {\it Let $a,b:~E\to\R$ be two
functions on the edges of a directed graph $G=(V,E)$. Then there is a
circulation $g:~E\to\R$ such that $a(ij)\le g(ij)\le b(ij)$ for every edge $ij$
if and only if $a\le b$ and
\[
\sum_{i\in A, j\in A^c} a(ij) \le \sum_{i\in A^c, j\in A} b(ij)
\]
for every $A\subseteq V$.}

\medskip

The most important consequence of the Hoffman Circulation Theorem is the
Max-Flow-Min-Cut Theorem of Ford and Fulkerson \cite{FF}. Let $s,t\in V$ and
let $c:~E\to\R_+$ be an assignment of nonnegative ``capacities'' to the edges.
An {\it $s$-$t$ cut} is a set of edges from $A$ to $A^c$, where $s\in A$ and
$t\notin A$. The {\it capacity} of this cut is the sum $\sum_{i\in A,\,j\in
A^c} c(ij)$.

An {\it $s$-$t$ flow} is function $f:~E\to\R$ satisfying the flow condition
\eqref{EQ:FLOW-COND} at every node $i~\not= s,t$. The {\it value} of the flow
is
\[
\val(f)=\sum_{j:\,sj\in E} f(sj) - \sum_{j:\,js\in E} f(js) =\sum_{j:\,jt\in E} f(jt) - \sum_{j:\,tj\in E} f(tj).
\]
A flow is {\it feasible}, if $0\le f\le c$.

\medskip

\noindent{\bf Max-Flow-Min-Cut Theorem.} {\it The maximum value of a feasible
$s$-$t$ flow is the minimum capacity of an $s$-$t$ cut. }

\medskip

Instead of specifying just two nodes, we can specify a {\it supply} and a {\it
demand} at each node, and require that the difference between the outflow and
the inflow be the difference between the supply and the demand.

Suppose that there is a circulation $g$ satisfying the given conditions
$a(e)\le g(e)\le b(e)$ for every (directed) edge $e$ (for short, a feasible
circulation). Also suppose that we are given a ``cost'' function $c:~E\to\R_+$.
What is the minimum of the ``total cost'' $\sum_e c(e)g(e)$ for a feasible
circulation? This can be answered by solving a linear program, where the
Duality Theorem applies; the condition is somewhat awkward, we'll state it
later for the general (measure) case.

Let $G=(V,E)$ be a (finite) directed graph. A {\it multicommodity flow} is a
family of flows $(f_{st}:~s,t\in V)$, where $f_{st}$ is a (nonnegative) $s$-$t$
flow. Suppose that we are given capacities $c(i,j)\ge 0$ for the edges and
demands $\sigma(s,t)\ge 0$ for all pairs of nodes. Then we say that the
multicommodity flow is {\it feasible}, if $f_{st}$ has value $\sigma(s,t)$, and
\[
\sum_{s,t} f_{st}(ij)\le c(i,j)
\]
for every edge $ij$. (We may assume, if convenient, that the graph is a
bidirected complete graph, since missing edges can be added with capacity $0$.)

The question is whether a feasible multicommodity flow exists. This is not
really hard, since the conditions can be written as a system of linear
inequalities, treating the values $f_{st}(i,j)$ as variables, and we can apply
Linear Programming. However, working out the dual we get conditions that are
not too transparent. But for undirected graphs there is a very nice form of the
condition due to Iri \cite{Iri} and to Shahroki and Matula \cite{ShMa}.

Let $G=(V,E)$ be an undirected graph, where we consider each undirected edge as
a pair of oppositely directed edges. Let us assume that the demand function
$\sigma(i,j)$ and the capacity function are symmetric:
$\sigma(i,j)=\sigma(j,i)$ and $c(i,j)=c(j,i)$. Consider a pseudometric $D$ on
$V$ (a function $D:~V\to V$ that is nonnegative, symmetric and satisfies the
triangle inequality, but $D(x,y)$ may be zero for $x\not=y$). If a feasible
multicommodity flow exists, then
\begin{equation}\label{EQ:ISM-COND}
\sum_{s,t\in V} \sigma(s,t) D(s,t) \le \sum_{ij\in E} c(i,j) D(i,j)
\end{equation}
(Just write each $s$-$t$ flow as a nonnegative linear combination of paths and
cycles, and use that the sum of edge lengths along each path is at least
$D(s,t)$.) We call this inequality the {\it volume condition}. When required
for every pseudometric, it is also sufficient:

\medskip

\noindent{\bf Multicommodity Flow Theorem.} {\it There exist a feasible
multicommodity flow satisfying the demands if and only if the volume condition
\eqref{EQ:ISM-COND} is satisfied for every pseudometric $D$ in $V$.}

\subsection{Graph limits}

\subsubsection{Graphons}

Let $(J,\AA)$ be a standard Borel space, and let $W:~J\times J\to[0,1]$ be a
measurable function. Let us endow $(J,\AA)$ with a {\it node measure}, a
probability measure $\lambda$. If $W$ is symmetric (i.e. $W(x,y)=W(y,x)$), then
the quadruple $(J,\AA,\lambda,W)$ is called a {\it graphon}. Dropping the
assumption that $W$ is symmetric, we get a {\it digraphon}.

The {\it edge measure} of a graphon or digraphon is the integral measure of
$W$,
\[
\eta(S) = \int\limits_S W\,d(\lambda\times\lambda).
\]
The node measure and edge measure of a graphon determine the graphon, up to a
set of $(\lambda\times \lambda)$-measure zero. Indeed, $\eta$ is absolutely
continuous with respect to $\lambda\times\lambda$, and $W=d\eta/d(\lambda\times
\lambda)$ almost everywhere.

Graphons can represent limit objects of sequences of dense graphs that are
convergent in the local sense \cite{BCLSV1,LSz1}. For this representation, we
may limit the underlying sigma-algebra to standard Borel spaces.

\subsubsection{Graphings}

Let $(J,\AA)$ be a standard Borel space. A {\it Borel graph} is a simple
(infinite) graph on node set $J$, whose edge set $E$ belongs to $\AA\times\AA$.
By ``graph'' we mean a simple undirected graph, so we assume that $E\subseteq
J\times J$ avoids the diagonal of $J\times J$ and is invariant under
interchanging the coordinates. A {\it graphing} is a Borel graph, with all
degrees bounded by a finite constant, endowed with a probability measure
$\lambda$ on $(J,\AA)$, satisfying the following ``measure-preservation''
condition for any two subsets $A,B\in\AA$:
\begin{equation}\label{EQ:UNIMOD}
\int\limits_A \deg_B(x)\,d\lambda(x) =\int\limits_B \deg_A(x)\,d\lambda(x).
\end{equation}
Here $\deg_B(x)$ denotes the number of edges connecting $x\in J$ to points of
$B$. (It can be shown that this is a bounded Borel function of $x$.) We call
$\lambda$ the {\it node measure} of the graphing.

We can define {\it Borel digraphs} (directed graphs) in the natural way, by
allowing $E$ to be any set in $\AA\times\AA$. To define a {\it digraphing}, we
assume that both the indegrees and outdegrees are finite and bounded. In this
case we have to define two functions: $\deg^+_B(x)$ denotes the number of edges
from $x$ to $B$, and $\deg^-_B(x)$ denotes the number of edges from $B$ to $x$.
The ``measure-preservation'' condition says that
\begin{equation}\label{EQ:UNIMOD-D}
\int\limits_A \deg^+_B(x)\,d\lambda(x) =\int\limits_B \deg^-_A(x)\,d\lambda(x)
\end{equation}
for $A,B\in\AA$. Such a digraphing defines a measure on Borel subsets of $J^2$,
the {\it edge measure} of the digraphing: on rectangles we define
\[
\eta(A\times B) = \int\limits_A \deg^+_B(x)\,d\lambda(x),
\]
which extends to Borel subsets in the standard way. This measure is
concentrated on the set of $E$ of edges. In the case of graphings, the edge
measure is symmetric in the sense that interchanging the two coordinates does
not change it. The node measure and the edge measure determine the (di)graphing
up to a set of edges of $\eta$-measure zero.

Graphings can represent limit objects of sequences of bounded-degree graphs
that are convergent in the local (Benjamini--Schramm) sense \cite{BSch,Elek1},
but also in a stronger, local-global sense \cite{HLSz}.

\subsubsection{Double measure spaces}

For both graphons and graphings, all essential information is contained in the
quadruple $(J,\AA,\lambda,\eta)$, where the {\it node measure} $\lambda$ is a
probability measure on $(J,\AA)$ and the {\it edge measure} $\eta$ is a
symmetric measure on $(J\times J,\AA\times\AA)$. Such a quadruple will be
called a {\it double measure space}. Graphons are those double measure spaces
where $\eta$ is dominated by $\lambda\times\lambda$; the function $W$
describing the graphon is the Radon-Nikodym derivative
$d\eta/d(\lambda\times\lambda)$. Graphings, on the other hand, are those double
measure spaces whose edge measure is extremely singular with respect to
$\lambda\times\lambda$.

It turns out that double measure spaces play a role in other recent work in
graph limit theory, as limit objects for graph sequences that are neither dense
nor bounded-degree, but convergent in some well-defined sense: shape
convergence \cite{KLSz} or action convergence \cite{BackSz}. We don't describe
these limit theories here, but as an example for which a very reasonable limit
can be defined in terms of double measure spaces we mention the sequence of
hypercubes.

We can scale the edge measure of a double measure space to get a probability
measure; if we drop the node measure (or restrict our interest to the case when
$\lambda$ is the marginal of $\eta$, to get to our main object of study, Markov
spaces. Except for the scaling factor, this generalizes regular graphs. To
construct limits of non-regular graphs we need the additional information
contained in the node measure; the marginal of $\eta$ corresponds to the degree
sequence.

\subsubsection{Markov spaces}\label{SEC:MARKOV}

A {\it Markov space} consists of a sigma-algebra $\AA$, together with a
probability measure $\eta$ on $\AA^2$ whose marginals are equal. We call $\eta$
the {\it ergodic circulation}, and its marginals $\pi=\eta^1=\eta^2$, the {\it
stationary distribution} of the Markov space $(\AA,\eta)$.

As the terminology above suggests, Markov spaces are intimately related to
Markov chains. To define a Markov chain, we need a sigma-algebra $\AA$ and a
probability measure $P_u$ on $\AA$ for every $u\in J$, called the {\it
transition distribution} from $u$. One assumes that for every $A\in \AA$, the
value $P_u(A)$ is a measurable function of $u\in J$. This structure is
sometimes called a {\it Markov scheme}.

If we also have a {\it starting distribution} on $(J,\AA)$, then we can
generate a {\it Markov chain}, i.e. a sequence of random points $(\wb^0, \wb^1,
\wb^2,\ldots)$ of $J$ such that $\wb^0$ is chosen from the starting
distribution, and $\wb^{i+1}$ is chosen from distribution $P_{\wb^i}$
(independently of the previous elements $\wb^0,\ldots,\wb^{i-1}$ of the Markov
chain). Sometimes we call this sequence a {\it random walk}.

A probability measure $\pi$ on $(J,\AA)$ is a {\it stationary distribution} for
the Markov scheme if choosing $\wb^0$ from this distribution, the next point
$\wb^1$ of the walk will have the same distribution. While finite Markov
schemes always have a stationary distribution, this is not true for infinite
underlying sigma-algebras. Furthermore, a Markov scheme may have several
stationary distributions. (In the finite case, this happens only if the
underlying directed graph is not strongly connected.)

A Markov scheme $(J,\{P_u:~u\in J\})$ with a fixed stationary distribution
$\pi$ defines a Markov space, whose ergodic circulation is the joint
distribution measure $\eta$ of $(\wb^0,\wb^1)$, where $\wb^0$ is a random point
from the stationary distribution. Both marginals of this ergodic circulation
equal to the stationary distribution $\pi$.

The ergodic circulation $\eta$ determines the Markov scheme (except for a set
of measure zero in the stationary measure). Using the Disintegration Theorem
(Proposition \ref{PROP:DISINT} below), one can show that every Markov space is
obtained by this construction from a Markov scheme with a stationary
distribution.

It is clear that if $(\AA,\eta)$ is a Markov space, then $(\AA,\eta^*)$ is a
Markov space with the same stationary distribution. The corresponding Markov
chain is called the {\it reverse chain}. A Markov space is {\it reversible}, if
$\eta=\eta^*$. A Markov space $(\AA,\eta)$ is {\it indecomposable}, if
$\eta(A\times A^c)>0$ for every set $A\in\AA$ with $0<\pi(A)<1$.

Flow problems on graphons and graphings can be formulated as flow problems on
double measure spaces; we'll see that many of them can be formulated as flow
problems on Markov spaces, without reference to the node measure. The solutions
we obtain yield solutions in the settings of graphings and graphons, via
Radon--Nikodym derivatives. However, as mentioned in the introduction, these
are just ``pure existence proofs'' (cf.\ also Remark \ref{REM:MEAS2FNC}).

\section{Auxiliaries}

\subsection{Measures}

Let $(J,\AA)$ be a sigma-algebra. Unless specifically emphasized otherwise, we
assume that $(J,\AA)$ is a standard Borel space of continuum cardinality; in
particular, $\AA$ is separating any two points, and it is countably generated.
Since the sigma-algebra $\AA$ determines its underlying set, we can talk about
the standard Borel space as a sigma-algebra (where, in the case of the
sigma-algebra denoted by $\AA$, the underlying set will be denoted by $J$). We
denote by $\ca(\AA)$ the linear space of finite signed (countably additive)
measures on $\AA$, and by $\ca_+(\AA)$, the set of nonnegative measures in
$\ca(\AA)$. We denote by $\delta_s$ the Dirac measure, the probability
distribution concentrated on $s\in J$.

If $\mu\in\ca(\AA)$ and $f:~J\to\R$ is a $\mu$-integrable function, then we
define a signed measure $f\cdot\mu\in\ca(\AA)$ and a number $\mu(f)$ by
\[
(f\cdot\mu)(A)=\int\limits_A f\,d\mu\quad(A\in\AA),\qquad \mu(f)=(f\cdot\mu)(J) = \int\limits_J f\,d\mu.
\]

We endow the linear space $\ca(\AA)$ with the total variation norm
\begin{equation}\label{EQ:TV-DEF}
\|\alpha\| = \sup_{A\in\AA}\alpha(A)-\inf_{B\in\AA}\alpha(B).
\end{equation}
We note that the supremum and the infimum are attained, when $J=A\cup B$ is a
Hahn decomposition of $\alpha$. With this norm, $\ca(\AA)$ becomes a Banach
space. This norm defines a metric on $\ca(\AA)$, the {\it total variation
distance}
\[
d_\tv(\alpha,\beta) = \|\alpha-\beta\|.
\]
Warning: if $\alpha$ and $\beta$ are probability measures, then
$\sup_{A\in\AA}(\alpha(A)-\beta(A))=-\inf_{A\in\AA}(\alpha(A)-\beta(A))$, and
so $d_\tv(\alpha,\beta)= 2 \sup_A (\alpha(A)-\beta(A))$. In probability theory,
the total variation distance is often defined as $\sup_A (\alpha(A)-\beta(A))$,
a factor of $2$ smaller.

For $\mu\in\ca(\AA)$ and $A\in \AA$, we define the restriction measure
$\mu_A\in\ca(\AA)$ by $\mu_A(X)=\mu(A\cap X)$. We denote the Jordan
decomposition of a signed measure $\alpha\in\ca(\AA)$ by
$\alpha=\alpha_+-\alpha_-$, and its total variation measure by
$|\alpha|=\alpha_++\alpha_-$. So $\|\alpha\|= \alpha_+(J)+\alpha_-(J) =
|\alpha|(J)$. For two measures $\alpha,\beta$ on $\AA$, we consider the Jordan
decomposition of their difference
$\alpha-\beta=(\alpha-\beta)_+-(\alpha-\beta)_-=(\alpha-\beta)_+-(\beta-\alpha)_+$,
and define the measures
\[
\alpha\setminus\beta= (\alpha-\beta)_+,\qquad
\alpha\land\beta = \alpha-(\alpha-\beta)_+ = \beta-(\beta-\alpha)_+.
\]
The measure $\alpha\land\beta$ is the largest nonnegative measure $\gamma$
dominated by both $\alpha$ and $\beta$.

If $\AA$ is a sigma-algebra, we denote by $\AA^2=\AA\times\AA$ the product
sigma-algebra of $\AA$ with itself; $\AA^3$ etc.~are defined analogously.
Sometimes it will be necessary to distinguish the factors (even though they are
identical), and we write $\AA^3=\AA_1\times\AA_2\times\AA_3 = \AA^{\{1,2,3\}}$
etc. For a measure $\mu\in\ca(\AA^n)$, and $T\subseteq \{1,\dots,n\}$, we let
$\mu^T$ denote its marginal on all coordinates in $T$. To simplify notation, we
write $\mu^{34}=\mu^{\{3,4\}}$, etc.

We need some further definitions for the sigma-algebra $\AA^2$ and for measures
on it. For $X\subseteq J\times J$, let $X^*=\{(x,y):~(y,x)\in X\}$. For a
function $f:~J\times J\to\R$, we define $f^*(x,y) = f(y,x)$. For a signed
measure $\mu$ on $\AA\times \AA$, we define $\mu^*(X)=\mu(X^*)$. A measure
$\mu$ on $J\times J$ that is {\it symmetric} if $\mu^*=\mu$.

We set $\mu^B(A)=\mu(A\times B)$. So $\mu^1=\mu^J$ and $\mu^2=(\mu^*)^J$ for
$\mu\in\ca(\AA^2)$. If $\mu^1=\lambda_1$ and $\mu^2=\lambda_2$, then we say
that $\mu$ is {\it coupling} the measures $\lambda_1$ and $\lambda_2$.

A {\it circulation} is a finite signed measure $\alpha\in\ca(\AA^2)$ with equal
marginals: $\alpha^1=\alpha^2$. Every symmetric measure is a circulation in a
trivial way. We'll return to circulations in the next section. We say that a
measure $\beta\in\ca_+(\AA^2)$ is {\it acyclic}, if there is no nonzero
circulation $\alpha$ such that $0\le\alpha\le\beta$. Every measure in
$\ca_+(\AA^2)$ can be written as the sum of a nonnegative acyclic measure and a
nonnegative circulation (this decomposition is not necessarily unique).

We need some well-known facts about measures.

\begin{lemma}\label{LEM:BOX-COMP}
Let $(J,\AA)$ be a standard Borel space, and $\psi\in\ca_+(\AA)$. Let
$\mu_1,\mu_2,\dots\in\ca(\AA)$ be signed measures with $|\mu_n|\le\psi$. Then
there is a subsequence $n_1<n_2<\dots$ of natural numbers and a signed measure
$\mu\in\ca(\AA)$ such that $|\mu|\le\psi$ and $\mu_{n_i}(A)\to\mu(A)$ for every
$A\in\AA$.
\end{lemma}

It follows easily that, more generally, $\mu_{n_i}(f)\to\mu(f)$ for every
bounded measurable function $f:~J\to\R$.

\begin{proof}
We may assume that $\mu_n\ge0$ (just add $\psi$ to every measure). Let $\BB$ be
a countable set algebra generating $\AA$. The sequence
$(\mu_n(B):~n=1,2,\dots)$ is bounded for every $B\in\BB$, so choosing an
appropriate subsequence, we may assume that there is a function $\mu:~\BB\to\R$
such that $\mu_n(B)\to\mu(B)$ for all $B\in\BB$. Clearly $\mu_n$ is a
pre-measure on $\BB$. We claim that $\mu$ is a pre-measure on $\BB$. Finite
additivity of $\mu$ is trivial, and so is $0\le\mu(B)\le\psi(B)$ for $B\in\BB$.
If $B_1\supseteq B_2\supseteq\dots$ $(B_i\in\BB)$ and $\cap_k B_k=\emptyset$,
then $\mu(B_k)\le \psi(B_k)$, and since $\psi(B_k)\to 0$ as $k\to\infty$, we
have $\mu(B_k)\to0$ as well.

It follows that $\mu$ extends to a measure on $\AA$. Uniqueness of the
extension implies that $0\le\mu\le\psi$ on the whole sigma-algebra $\AA$. Let
$S\in\AA$; we claim that $\mu_n(S)\to\mu(S)$ ($n\to\infty$). For every
$\eps>0$, there is a set $B\in\BB$ such that $\psi(B\triangle A)\le\eps/3$.
This implies that $|\mu_n(S)-\mu_n(B)|\le \mu_n(S\triangle B)\le\psi(S\triangle
B)\le\eps/3$, and similarly $|\mu(S)-\mu(B)|\le \eps/3$. Thus
$|\mu_n(S)-\mu(S)|\le |\mu_n(B)-\mu(B)|+2\eps/3$. Since $\mu_n(B)\to\mu(B)$ by
the definition of $\mu$, we have $|\mu_n(S)-\mu(S)|\le\eps$ if $n$ is large
enough.
\end{proof}

The following fact follows by a very similar argument.

\begin{lemma}\label{LEM:COUPLE-CONV}
Let $(J,\AA)$ be a standard Borel space, and let $\lambda_1,\lambda_2$ be
probability measures on $(J,\AA)$. Let $\mu_n\in\ca(\AA^2)$ $(n=1,2,\dots)$ be
measures coupling $\lambda_1$ and $\lambda_2$. Then there is an infinite
subsequence $\mu_{n_1},\mu_{n_2},\dots$ and a measure $\mu$ coupling
$\lambda_1$ and $\lambda_2$ such that $\mu_{n_i}(A\times B)\to\mu(A\times B)$
for all sets $A,B\in \AA$.\proofend
\end{lemma}

We need a special version of the important construction of {\it
disintegration}; see \cite{DelM,ChPol,Kech,Bog} for more details.

\begin{prop}\label{PROP:DISINT}
Let $(J,\AA)$ be a standard Borel space, and let $\psi\in\ca(\AA\times\AA)$.
Then there is a family of signed measures $\varphi_x\in\ca(\AA)$ $(x\in J)$
such that $\varphi_x(A)$ is a measurable function of $x$ for every $A\in\AA$,
and
\[
\psi(B)=\int\limits_J \varphi_x(B\cap (\{x\}\times J))\,d\varphi^1(x)
\]
for every $B\in\AA^2$.\proofend
\end{prop}

One can think of $\varphi_x$ as $\psi$ conditioned on $\{x\}\times J$, even
though the condition has (typically) probability $0$, and so the conditional
probability in the usual sense is not defined.

\subsection{Linear functionals}

We need some simple facts of Banach space theory; for completeness, we include
their simple derivations from standard results.

\begin{lemma}\label{LEM:HB-KN}
Let $K_1,\dots,K_n$ be open convex sets in a Banach space $B$. Then
$K_1\cap\dots\cap K_n=\emptyset$ if and only if there are bounded linear
functionals $\LL_1,\dots\LL_n$ on $B$ and real numbers $a_1,\dots,a_n$ such
that $\LL_1+\dots+\LL_n=0$, $a_1+\dots+a_n=0$, and for each $i$, either
$\LL_i=0$ and $a_i=0$, or $\LL_i(x)>a_i$ for $x\in K_i$, and for at least one
$i$, the second possibility holds.\proofend
\end{lemma}

If $\LL_i=0$ and $a_i=0$ for some $i$, then already the intersection of the
sets $K_j$ $(j\not=i)$ is empty.

\begin{proof}
The sufficiency of the condition is trivial. To prove the necessity, consider
the Banach space $B'=B\oplus\dots\oplus B$ ($n$ copies) and the open convex set
$K'=K_1\times\dots\times K_n\subseteq B'$. If any $K_i$ is empty, then the
conclusion is trivial, so suppose that $K'\not=\emptyset$. Also consider the
closed linear subspace (``diagonal'') $\Delta =\{(x,\dots,x): ~x\in
B\}\subseteq B'$. Then $\Delta\cap B'=\emptyset$. By the Hahn--Banach Theorem,
there is a bounded linear functional $\LL$ on $B'$ such that $\LL(y)=0$ for
$y\in\Delta$, and $\LL(y)>0$ for $y\in K'$.

Define $\LL_i(x) = \LL(0,\dots,0,x,0,\dots,0)$ and $a_i=\inf_{x\in
K_i}\LL_i(x)$. Then $L_i$ is a bounded linear functional on $B$, and
$\LL(x_1,\dots,x_n)=\LL_1(x_1)+\dots+\LL_n(x_n)$. The condition that $\LL(y)=0$
for $y\in \Delta$ means that $\LL_1(x)+\dots+\LL_n(x)=0$ for all $x\in B$. For
each $i$, either $\LL_i=0$ and $a_i=0$, or $\LL_i(x)>a_i$ for $i\in K_i$ (as
$K_i$ is open). Since $\LL(y)>0$ for $y\in K'$, there must be at least one $i$
with $\LL_i\not=0$. Furthermore, $a_1+\dots+a_n = \inf_{y\in K'}\LL(y)\ge 0$.
We can decrease any $a_i$ to get equality in the last inequality.
\end{proof}

\begin{prop}\label{PROP:TT-INV}
Let $B_1$ and $B_2$ be Banach spaces and $\TT:~B_1\to B_2$, a bounded linear
transformation whose range is closed in $B_2$. Let $\LL:~B_1\to\R$ be a bounded
linear functional. Then $\LL$ vanishes on $\Ker(\TT)$ if and only if there is a
bounded linear functional $\KK:~B_2\to\R$ such that $\LL=\KK\circ\TT$.\proofend
\end{prop}

\begin{proof}
The ``if'' direction is trivial. To prove the converse, note that $\Ker(\TT)$
is a closed linear subspace of $B_1$, and so $B_0=B_1/\Ker(\TT)$ is a well
defined Banach space. The maps $\TT$ and $\LL$ induce bounded linear maps
$\TT_0:~B_0\to B_2$ and $\LL_0:~B_0\to\R$ (since $\LL$ vanishes on
$\Ker(\TT)$). Furthermore, $\TT_0$ is bijective. Since $\Rng(\TT_0)=\Rng(\TT)$
is closed in $B_2$ and therefore a Banach space, the Inverse Mapping Theorem
implies that $\TT_0^{-1}$ is bounded. So we can define $\KK$ on $\Rng(\TT)$ by
$\KK(x)=\LL_0(\TT_0^{-1}(x))$. By the Hahn--Banach Theorem, $\KK$ can be
extended to $B_2$.
\end{proof}

We will need linear functionals on the Banach space of measures. These
functionals do not seem to have a useful complete description, but the
following fact is often a reasonable substitute.

\begin{prop}\label{PROP:MEAS-REP}
Let $\LL$ be a bounded linear functional on $\ca(\AA)$ and $\psi\in\ca_+(\AA)$.
Then there is a bounded measurable function $g:~J\to\R$ such that
$\LL(\mu)=\mu(g)$ for every $\mu\in\ca(\AA)$ with $\mu\ll\psi$.\proofend
\end{prop}

\begin{proof}
We define a functional $\NN:~L_1(\AA,\psi)\to\R$ by $\NN(f)=\LL(f\cdot\psi)$
for $f\in L_1(\AA,\psi)$. Then $\NN$ is a bounded linear functional on
$L_1(\AA,\psi)$, and so there is a bounded measurable function $g$ on $(J,\AA)$
such that $\NN(f)=\psi(fg)$ for all $f\in L_1(J,\psi^J)$.

The condition that $\mu\ll\psi$ implies that the Radon-Nikodym derivative
$h=d\mu/d\psi\in L_1(\AA,\psi)$ exists, and $h\cdot\psi=\mu$. Thus
\[
\LL(\mu) =\NN(h) =\int\limits_J \frac{d\mu}{d\psi} g\,d\psi = \mu(g).\qedhere
\]
\end{proof}

We conclude with a technical lemma.

\begin{lemma}\label{LEM:FUNCT-SUP}
Let $\LL$ be a bounded linear functional on $\ca(\AA^2)$. Then there is a
bounded linear functional $\QQ$ on $\ca(\AA)$ such that for all
$\psi\in\ca_+(\AA)$,
\[
\QQ(\psi)=\sup\{\LL(\mu):~\mu\in\ca_+(\AA^2),~\mu^1=\psi\}.
\]
\end{lemma}

\begin{proof}
The formula in the lemma defines a functional on $\ca_+(\AA^2)$; we start with
showing that this is bounded and linear on nonnegative measures. For every
$\mu\in\ca_+(\AA^2)$ with $\mu^1=\psi$, we have $\|\mu\|=\|\psi\|$, and so
$\LL(\mu) \le \|\LL\|\,\|\mu\|=\|\LL\|\,\|\psi\|$. Thus $\QQ(\psi) \le
\|\LL\|\,\|\psi\|$. It is also clear that $\QQ(c\psi)=c\QQ(\psi)$ for $c>0$.

Let $\psi = \psi_1+\psi_2$ ($\psi_i\in\ca_+(\AA)$); we claim that
\begin{equation}\label{EQ:Q-SUM}
\QQ(\psi)= \QQ(\psi_1)+\QQ(\psi_2).
\end{equation}
For $\eps>0$, choose $\mu_i\in\ca_+(\AA^2)$, so that $\mu_i^1=\psi_i$ and
$\LL(\mu_i)\ge \QQ(\psi_i)-\eps$. Then
\[
\QQ(\psi) \ge \LL(\mu_i+\mu_2) = \LL(\mu_1)+\LL(\mu_2) \ge \QQ(\psi_1)+\QQ(\psi_2)-2\eps.
\]
Since this holds for every $\eps>0$, this proves that $\QQ(\psi)\ge
\QQ(\psi_1)+\QQ(\psi_2)$. To prove the reverse inequality, let
$\mu\in\ca_+(\AA^2)$ with $\mu^1=\psi$. Define the measures
\[
\mu_i(U)=\int\limits_U \frac{d\psi_i}{d\psi}(x) \,d\mu(x,y) \qquad(U\in\AA^2).
\]
It is easy to check that
\begin{equation}\label{EQ:MU-PSI}
\mu_1+\mu_2=\mu,  \quad\text{and}\quad \mu_i^1=\psi_i\quad(i=1,2).
\end{equation}
It follows that
\[
\LL(\mu) =\LL(\mu_1)+\LL(\mu_2) \le \QQ(\psi_1)+\QQ(\psi_2).
\]
Since this holds for every $\mu\in\ca_+(\AA^2)$ with $\mu^1=\psi$, we get that
$\QQ(\psi)\le\QQ(\psi_1)+\QQ(\psi_2)$. This implies \eqref{EQ:Q-SUM}.

Thus $\QQ$ is nonnegative, positive homogeneous and linear on $\ca_+(\AA)$. We
extend it to $\ca(\AA)$ by $\QQ(\mu)=\QQ(\mu_+)-\QQ(\mu_-)$. In particular, if
$\mu\le 0$, then $\QQ(\mu)=-\QQ(-\mu)$. This implies that the extended $\QQ$ is
homogeneous.

Let $\varphi,\psi\in\ca(\AA^2)$; we claim that
\begin{equation}\label{EQ:MUNU}
\QQ(\varphi+\psi) = \QQ(\varphi)+\QQ(\psi).
\end{equation}
We know that this holds if $\varphi,\psi\ge 0$, and it follows that it holds if
$\varphi,\psi\le0$. If $\varphi\ge0$, $\psi\le0$, and $\varphi+\psi\ge 0$, then
$\QQ(\varphi) = \QQ(\varphi+\psi)+\QQ(-\psi) = \QQ(\varphi+\psi)-\QQ(\psi)$, so
\eqref{EQ:MUNU} holds true. This implies easily that \eqref{EQ:MUNU} holds
whenever neither one of $\varphi$, $\psi$ and $\varphi+\psi$ changes sign.

To verify the general case, we consider the common refinement of the Hahn
decompositions for $\varphi$, $\psi$ and $\varphi+\psi$. We get a partition
$\PP$ into at most $8$ parts, where neither one of $\varphi$, $\psi$ and
$\varphi+\psi$ changes sign on any partition class. Then
\begin{align*}
\QQ(\varphi) &= \QQ(\varphi_+)-\QQ(\varphi_-)
= \sum_{X\in\PP:\,\varphi_X\ge 0} \QQ(\varphi_X) - \sum_{X\in\PP:\,\varphi_X\le 0} \QQ((\varphi_-)_X)
= \sum_{X\in\PP} \QQ(\varphi_X).
\end{align*}
(Note: \eqref{EQ:MUNU} has been applied to the restrictions of $\varphi$ to
subsets of the positive support, and separately to subsets of the negative
support.) Similarly,
\[
\QQ(\psi) = \sum_{X\in\PP} \QQ(\psi_X), \quad\text{and} \quad
\QQ(\varphi+\psi) = \sum_{X\in\PP} \QQ((\varphi+\psi)_X).
\]
Since we know already that $\QQ((\varphi+\psi)_X)=\QQ(\varphi_X)+\QQ(\psi_X)$,
this proves that $\QQ$ is additive.

Clearly $|\QQ(\varphi)|\le |\QQ(\varphi_+)|+|\QQ(\varphi_-)|\le
2\|\LL\|\,\|\varphi\|$, so $\QQ$ is continuous.
\end{proof}

\section{Potentials, circulations and flows}

\subsection{Potentials}

Let $(J,\AA)$ be a measurable space. A measurable function $F:~J\times J\to\R$
is a {\it potential}, if there is a measurable function $f:~ J\to\R$ such that
$F(x,y)=f(x)-f(y)$. It is easy to see that a bounded measurable function
$F:~J\times J\to\R$ is a potential if and only if $F(x,y)+F(y,z)+F(z,x)=0$ for
all $x,y,z\in J$.

Of particular importance will be {\it cut potentials} of the form
$\one_A(x)-\one_A(y)=\one_{A\times A^c}(x,y)-\one_{A^c\times A}(x,y)$, where
$A\in\AA$. Every potential $F$ can be expressed by cut potentials as
\begin{equation}\label{EQ:POTENT}
F(x,y) = \int\limits_{-C}^C (\one_{A_t}(x)-\one_{A_t}(y))\,dt,
\end{equation}
where $C$ is an upper bound on $|F|$, and $A_t$ ($-C\le t\le C$) is a
measurable subset of $J$ such that $A_t\subseteq A_s$ for $t<s$, $\cap_t
A_t=\emptyset$ and $\cup_t A_t=J$. To see this, let $F(x,y)=f(x)-f(y)$ for some
bounded measurable function $f$, and define $A_t=\{x\in J:~f(x)\ge t\}$ $(-C\le
t\le C)$.

\subsection{Circulations}

\subsubsection{Circulations and potentials}

Recall that $\alpha\in\ca(\AA^2)$ is a circulation if its two marginals
$\alpha^1$ and $\alpha^2$ are equal. This is clearly equivalent to saying that
\begin{equation}\label{EQ:CIRC1}
\alpha(X\times X^c)=\alpha(X^c\times X)\qquad(\forall X\in\AA)
\end{equation}
(just cancel the common part $X\times X$ in $\alpha(X\times J)=\alpha(J\times
X)$). Circulations form a linear subspace $\Cf=\Cf(\AA)$ of the space
$\ca(\AA^2)$ of finite signed measures.

In the finite case, circulations of the form
$\delta_{x_1x_2}+\dots+\delta_{x_{n-1}x_n}+\delta_{x_nx_1}$ generate the space
of all circulations (even those with $n\le 3$ do). In the measure case, this is
not always so, as the next example shows.

\begin{example}[Cyclic graphing and digraphing]\label{EXA:CA}
For a fixed $a\in(0,1)$, let $\Cb_a$ be the graphing on $[0,1]$ obtained by
connecting every point $x$ to $x+a\pmod1$ and $x-a\pmod1$. If $a$ is
irrational, this graph consists of two-way infinite paths; if $a$ is rational,
the graph will consist of cycles. We will also use the directed version
$\overrightarrow{C}_a$, obtained by connecting $x$ to $x+a\pmod1$ by a directed
edge.

The uniform measure $\mu$ on the edges of $\overrightarrow{C}_a$ is trivially a
circulation, both of its marginals being the uniform measure $\lambda$ on
$[0,1)$. Every circulation $\alpha$ supported on the edges is a constant
multiple of this. Indeed, $\alpha^1(A)=\alpha(A\times
(A+a))=\alpha^2(A+a)=\alpha^1(A+a)$ for every Borel set $A\subseteq[0,1)$,
which means that $\alpha^1$ is invariant under translation by $a$. It is
well-known that only scalar multiples of $\lambda$ have this property.
\end{example}

We need two lemmas describing ``duality'' relations between potentials and
circulations.

\begin{lemma}\label{LEM:CIRC-POT}
A signed measure $\alpha\in\ca(\AA^2)$ is a circulation if and only if
$\alpha(F)=0$ for every potential $F$.
\end{lemma}

\begin{proof}
The ``if'' part follows by applying the condition to the potential
$\one_{A}(x)-\one_{A}(y)$:
\[
\alpha(A\times J)-\alpha(J\times A)=
\int\limits_{J\times J}(\one_{A}(x)-\one_{A}(y))\,d\alpha(x,y) =0.
\]
To prove the converse, let $\alpha$ be a circulation, then for every potential
$F(x,y)=f(x)-f(y)$, we have
\[
\alpha(F)=\int\limits_{J\times J}f(x)-f(y)\,d\alpha(x,y)
= \int\limits_J f(x)\,d\alpha^2(x)-\int\limits_J f(y)\,d\alpha^1(y)=0.\qedhere
\]
\end{proof}

\begin{lemma}\label{LEM:CIRC-DUAL}
Let $\LL:~\ca(\AA^2)\to\R$ be a continuous linear functional. Then $\LL$
vanishes on the space $\Cf$ of circulations if and only if there is a
continuous linear functional $\KK:~\ca(\AA)\to\R$ such that
$\LL(\mu)=\KK(\mu^1-\mu^2)$ for all $\mu\in\ca(\AA^2)$.
\end{lemma}

\begin{proof}
The kernel of the linear operator $\varphi\mapsto\varphi^1-\varphi^2$
($\varphi\in\ca(\AA^2)$) is $\Cf$. The range of this operator is
\begin{equation}\label{EQ:NM}
\Rng(\TT)=\{\nu\in\ca(\AA):~\nu(J)=0\}.
\end{equation}
Indeed, if $\nu=\mu^1-\mu^2\in\Rng(\TT)$, then $\nu(J)=\mu(J\times
J)-\mu(J\times J)=0$. Conversely, if $\nu(J)=0$, then for any probability
measure $\gamma$ on $\AA$,
\[
\TT(\gamma\times\nu) = \gamma(J)\nu -\nu(J)\gamma = \nu,
\]
so $\nu$ is in the range of $\TT$. It is easy to check that $\nu(J)=0$ defines
a closed subspace of $\ca(\AA)$. Hence Proposition \ref{PROP:TT-INV} implies
the necessity of the condition. The sufficiency is straightforward, since
$\mu^1-\mu^2=0$ for every circulation $\mu$.
\end{proof}

Let $\LL\in\Cf^\perp$ and $\psi\in\ca_+(\AA^2)$. Restricting $\LL$ to measures
$\mu\ll\psi$, we get a more explicit representation: there is a potential $F$
such that
\begin{equation}\label{EQ:LL-PSI}
\LL(\mu)=\mu(F)\qquad(\mu\ll\psi).
\end{equation}
Indeed, consider the continuous linear functional $\KK$ constructed in Lemma
\ref{LEM:CIRC-DUAL}, and its representation $\KK(\nu)=\nu(g)$ by a bounded
measurable function $g:~J\to\R$ in Proposition \ref{PROP:MEAS-REP}, valid for
every $\nu\ll\psi^1+\psi^2$. Then for the potential $F(x,y)=g(x)-g(y)$ and
every $\mu\ll\psi$,
\begin{align*}
\mu(F) &=\int\limits_{J\times J} g(x)-g(y)\,d\mu(x,y) = \mu^1(g)-\mu^2(g) = \KK(\mu^1-\mu^2)=\LL(\mu).
\end{align*}

\subsubsection{Existence of circulations}

Now we begin to carry out our program of extending basic flow-theoretic results
in combinatorial optimization to measures. Our first goal is to generalize the
Hoffman Circulation Theorem and to characterize optimal circulations.

Given two measures $\varphi$ and $\psi$ on $J\times J$, we can ask whether
there exists a circulation $\alpha$ such that $\varphi\le \alpha\le \psi$.
Clearly $\varphi\le\psi$ is a necessary condition, but it is not sufficient in
general. The following theorem generalizes the Hoffman Circulation Theorem.

\begin{theorem}\label{THM:HOFF-CIRC-M}
For two signed measures $\varphi,\psi\in\Mf(J\times J)$, there exists a
circulation $\alpha$ such that $\varphi\le \alpha\le \psi$ if and only if
$\varphi\le\psi$ and $\varphi(X\times X^c)\le \psi(X^c\times X)$ for every set
$X\in\AA$.
\end{theorem}

\begin{proof}
The necessity of the condition is trivial: if the circulation $\alpha$ exists,
then $\varphi(X\times X^c)\le \alpha(X\times X^c)=\alpha(X^c\times X) \le
\psi(X^c\times X)$.

To prove sufficiency, consider the set
$\Xf=\{\mu\in\ca(\AA^2):~\varphi\le\mu\le\psi\}$. We may assume (by adding a
sufficiently large circulation, say $|\varphi|+|\varphi|^*$) that
$0\le\varphi\le\psi$. We want to prove that $\Cf\cap\Xf\not=\emptyset$.

First, we prove the weaker fact that
\begin{equation}\label{EQ:WEAKER}
d_{\tv}(\Cf,\Xf)=0.
\end{equation}
Suppose that $c=d_{\tv}(\Cf,\Xf)>0$. Let $\Xf'=\{\mu\in\ca(\AA^2):
~d_\tv(\mu,\Xf)<c\}$, then $\Xf'$ is a convex open subset of $\ca(\AA^2)$.
Since $\Xf'\cap\Cf=\emptyset$, the Hahn--Banach Theorem implies that there is a
bounded linear functional $\LL$ on $\ca(\AA^2)$ such that $\LL(\mu)=0$ for all
$\mu\in \Cf$, and $\LL(\mu)<0$ for all $\mu$ in the interior of $\Xf'$, in
particular for every $\mu\in\Xf$.

The first condition on $\LL$ implies, by representation \eqref{EQ:LL-PSI}, that
there is a potential function $F(x,y)=g(x)-g(y)$ (with a bounded and measurable
function $g:~J\to\R$) such that $\LL(\mu)=\mu(F)$ for every $\mu\in\ca(\AA^2)$
such that $\mu\ll\psi$. Let $|g|\le C$.

Let $S=\{(x,y):~g(x)>g(y)\}$ and $A_t=\{x\in J:~g(x)\ge t\}$. Clearly
$A_t\times A_t^c\subseteq S$ and $A_t^c\times A_t\subseteq S^c$. We can write
\[
g(x) = \int\limits_{-C}^C \one_{A_t}(x)\,dt,
\]
then
\begin{equation}\label{EQ:LG3}
\LL(\mu) = \int\limits_{-C}^C \int\limits_{J\times J} \one_{A_t}(x)-\one_{A_t}(y)\,d\mu(x,y)\,dt =
\int\limits_{-C}^C \mu(A_t\times A_t^c)-\mu(A_t^c\times A_t)\,dt.
\end{equation}
Let us apply this formula with $\mu(X)=\varphi(X\cap S)+\psi(X\setminus S)$.
Then
\[
\LL(\mu) = \int\limits_{-C}^C \mu(A_t\times A_t^c)-\mu(A_t^c\times A_t)\,dt
=\int\limits_{-C}^C \psi(A_t\times A_t^c)-\varphi(A_t^c\times A_t)\,dt \ge 0
\]
by hypothesis. On the other hand, we have $\varphi\le\mu\le\psi$, so
$\mu\in\Xf$, so $\LL(\mu)<0$. This contradiction proves \eqref{EQ:WEAKER}.

To conclude, we select circulations $\alpha_n\in\Cf$ and measures
$\beta_n\in\Xf$ such that $\|\alpha_n-\beta_n\|\to0$ ($n\to\infty$). By Lemma
\ref{LEM:BOX-COMP}, there is a measure $\beta\in\Xf$ such that
$\beta_n(S)\to\beta(S)$ ($n\to\infty$) for all $S\in\AA^2$ and an appropriate
subsequence of the indices $n$. Hence
\[
|\alpha_n(S)-\beta(S)|\le|\alpha_n(S)-\beta_n(S)|+|\beta_n(S)-\beta(S)|\le
\|\alpha_n-\beta_n\|+|\beta_n(S)-\beta(S)|\to0.
\]
In particular, for every $A\in\AA$ we have
\[
0=\alpha_n(A\times A^c)-\alpha_n(A^c\times A)\to\beta(A\times A^c)-\beta(A^c\times A),
\]
and so $\beta$ is a circulation, and by a similar argument, $\beta\in\Xf$.
\end{proof}

\begin{remark}\label{REM:MEAS2FNC}
As long as we restrict our attention to circulations $\alpha$ that are
absolutely continuous with respect to a given measure $\psi\in\ca_+(\AA^2)$, we
can define them as functions, considering the Radon--Nikodym derivative
$f=d\alpha/d\psi$. Then $f$ is a $\psi$-integrable function satisfying
\[
\int\limits_{A\times A^c} f\,d\psi = \int\limits_{A^c\times A} f\,d\psi
\]
for all $A\in\AA$. The value $f(x,y)$ can be interpreted as the flow value on
the edge $xy$. The marginals of $\alpha$, meaning the flow in and out of a
point, could also be defined using a disintegration of $\psi$. However, this
definition of circulation would depend on the measure $\psi$, while our
definition above does not depend on any such parameter.

Similar remarks apply to notions like flows below, and will not be repeated.
\end{remark}

\subsubsection{Optimal circulations}

If a feasible circulation exists, we may be interested in finding a feasible
circulation $\mu$ which minimizes a ``cost'', or maximizes a ``value''
$\mu(v)$, given by a bounded measurable function $v$ on $J\times J$.
Equivalently, we want to characterize when a value of $1$ (say) can be
achieved. This cannot be characterized in terms of cut conditions any more, but
an elegant necessary and sufficient condition can still be formulated.

\begin{theorem}\label{THM:MAX-CIRC}
Given a bounded measurable function $v:~J\times J\to \R_+$ and measures
$\varphi,\psi\in\ca_+(\AA^2)$, $\varphi\le\psi$, there is a circulation
$\alpha$ with $\varphi\le \alpha\le\psi$ and $\alpha(v)=c$ if and only if the
following three conditions are satisfied for every potential $F$:
\begin{align}
&\psi(|F+v|_+) \ge \varphi(|F+v|_-) + c,\label{JJFB1}\\
&\psi(|F-v|_+) \ge \varphi(|F-v|_-) - c,\label{JJFB2}\\
&\psi(|F|_+) \ge \varphi(|F|_-)\label{JJFB3}
\end{align}
\end{theorem}

Condition \eqref{JJFB3} is equivalent to the condition given for the existence
of a circulation in Theorem \ref{THM:HOFF-CIRC-M}, which is obtained when
$F(x,y)=\one_X(x)-\one_X(y)$. If $\varphi=0$, then only \eqref{JJFB1} is
nontrivial. Applying the conditions with $F=0$ we get that $\varphi(v)\le c\le
\psi(v)$.

\begin{proof}
We may assume that $c=1$. The necessity of the condition is trivial: if such a
circulation $\alpha$ exists, then
\[
\psi(|F+v|_+) - \varphi(|F+v|_-) \ge \alpha(|F+v|_+) - \alpha(|F+v|_-)= \alpha(F+v) = \alpha(v) = 1,
\]
and similar calculation proves the other two conditions.

To prove the converse, we proceed along similar lines as in the proof of
Theorem \ref{THM:HOFF-CIRC-M}. Consider the subspace $\Cf\subseteq \ca(\AA^2)$
of circulations, the affine hyperplane $\Hf=\{\alpha\in\ca(\AA^2):~
\alpha(v)=1\}$ and the ``box'' $\Xf=\{\alpha\in\ca(\AA^2):~ \varphi\le\alpha\le
\psi\}$. We want to prove that $\Cf\cap\Hf\cap\Xf\not=\emptyset$.

Clearly the sets $\Cf$, $\Hf$ and $\Xf$ are nonempty. Fix an $\eps>0$, and
replace them by their $\eps$-neighborhoods
$\Cf'=\{\mu\in\ca(\AA^2):~d_\tv(\mu,\Cf)<\eps\}$ etc. We start with proving the
weaker statement that
\begin{equation}\label{EQ:CHX-WEAKER}
\Cf'\cap \Hf'\cap \Xf'\not=\emptyset.
\end{equation}
Suppose not. Then Lemma \ref{LEM:HB-KN} implies that there are bounded linear
functionals $\LL_1,\LL_2,\LL_3$ on $\ca(\AA^2)$, not all zero, and real numbers
$a_1,a_2,a_3$ such that $\LL_1+\LL_2+\LL_3=0$, $a_1+a_2+a_3=0$, and
$\LL_i(\mu)\ge a_i$ for all $\mu\in \Cf'$, $\Hf'$ and $\Xf'$, respectively, and
$\LL_i(\mu)>a_i$ for at least one $i$.

The functional $\LL_1$ remains bounded from below for every circulation
$\alpha\in\Cf$, and since $\Cf$ is a linear subspace, this implies that
\begin{equation}\label{EQ:LL1}
\LL_1(\alpha)=0 \qquad(\alpha\in\Cf).
\end{equation}
By a similar reasoning, $\LL_2$ must be a constant $b$ on the hyperplane $\Hf$;
we may scale $\LL_1$, $\LL2$ and $\LL_3$ so that $b\in\{-1,0,1\}$. It is easy
to see that this implies the more general formula
\begin{equation}\label{EQ:LL2MU}
\LL_2(\mu) = b\mu(v)\qquad(\mu\in\ca(\AA^2)),
\end{equation}
Finally, we can express $\LL_3$ as
\begin{equation}\label{EQ:LL3}
\LL_3(\mu)=-\LL_1(\mu)-\LL_2(\mu)\quad(\mu\in\ca(A^2)).
\end{equation}
Using the representation \eqref{EQ:LL-PSI}, we can write
\begin{equation}\label{EQ:LL1MU}
\LL_1(\mu) = \mu(F)\qquad(0\le \mu\le\psi)
\end{equation}
with some potential $F$ on $J\times J$. Hence
\[
\LL_3(\mu)=-\mu(F)-b\mu(v)=-\mu(F+bv) \qquad(0\le \mu\le \psi).
\]
We also know that for any $\alpha\in\Cf$, $\nu\in\Hf$ and $\mu\in\Xf$, we have
\[
0=a_1+a_2+a_3<\LL_1(\alpha)+\LL_2(\nu)+\LL_3(\mu)
= 0+b+\LL_3(\mu) = b-\mu(F+bv),
\]
and hence $\mu(F+bv)<b$ for all $\mu\in\Xf$.

The tightest choice for $\mu\in\Xf$ is $\mu=\psi_U-\varphi_{U^c}$, where
$U=\{(x,y):~F(x,y)+bv(x,y)\ge 0\}$. This gives that
\[
\psi(|F+bv|_+) - \varphi(|F+bv|_-) = \psi_U(F+bv) - \varphi_{U^c}(F+bv) = \mu(F+bv) <b.
\]
This contradicts one of the conditions in the theorem (depending on $b$). This
proves \eqref{EQ:CHX-WEAKER}.

To prove the stronger statement that $\Cf\cap\Hf\cap\Xf\not=\emptyset$,
\eqref{EQ:CHX-WEAKER} implies that there are sequences of measures $\alpha_n\in
\Cf$, $\nu_n\in\Hf$ and $\mu_n\in\Xf$ such that $d_\tv(\mu_n,\alpha_n)\to0$ and
$d_\tv(\mu_n,\nu_n)\to0$. Furthermore, since $0\le\mu_n\le\psi$, Lemma
\ref{LEM:BOX-COMP} applies, and so there is a measure $\mu\in\Xf$ such that for
an appropriate infinite subsequence of indices, $\mu_n(U)\to\mu(U)$ for all
$U\in\AA^2$. This implies that $\alpha_n(U)\to\mu(U)$ and $\nu_n(U)\to\mu(U)$
for this subsequence.

Thus
\[
\mu(A\times A^c) = \lim_{n\to\infty} \alpha_n(A\times A^c) =
\lim_{n\to\infty} \alpha_n(A^c\times A) = \mu(A^c\times A)
\]
for every $A\in\AA$, so $\mu\in\Cf$. Similarly, by Lemma \ref{LEM:BOX-COMP}
$\mu(v) = \lim_{n\to\infty}\nu_n(v) =1$, whence $\mu\in\Hf$.
\end{proof}

A straightforward application of Theorem \ref{THM:MAX-CIRC} allows us to answer
a question about the existence of Markov spaces, where an upper bound on the
ergodic circulation is prescribed.

\begin{corollary}\label{COR:ERG-FLOW-M}
Given a measure $\psi\in\ca_+(\AA^2)$, there exists an ergodic circulation
$\eta$ such that $\eta\le \psi$ if and only if every potential $F:~J\times
J\to\R$ satisfies
\[
\psi(|1+F|_+) \ge 1.
\]
\end{corollary}

\subsubsection{Integrality}\label{SSEC:INTEGER}

In the case when $v\equiv 1$ and $\varphi\equiv 0$, the condition in Corollary
\ref{COR:ERG-FLOW-M} implies that
\[
\psi(A\times A^c)-\psi(A^c\times A)\le \psi(J\times J)-1 \qquad (A\in\AA).
\]
One may wonder whether, at least in this special case, such a cut condition is
also sufficient in Corollary \ref{COR:ERG-FLOW-M}. This, however, fails even in
the finite case: on the directed path of length 2 where the edges have capacity
$1$, these cut conditions for the existence of an ergodic circulation are
satisfied, but the only feasible circulation is the $0$-circulation.

However, the following weaker requirement can be imposed on $F$:

\begin{suppl}\label{PROP:F-INT}
In Theorem \ref{THM:MAX-CIRC}, if the function $v$ has only integral values,
then it suffices to require condition \eqref{JJFB1}--\eqref{JJFB3} for
potentials $F$ having integral values.
\end{suppl}

This property of $F$ is clearly equivalent to saying that in the representation
$F(x,y)=f(x)-f(y)$, the function $f$ can be required to have integral values.
For finite graphs, this assertion follows easily from the fact that the matrix
of flow conditions is totally unimodular. In the infinite case, we have to use
another proof.

\begin{proof}
Suppose that there is a potential $F(x,y)=f(x)-f(y)$ violating (say)
\eqref{JJFB1}. Let $S=\{(x,y):~F(x,y)+v(x,y)>0\}$. Consider the modified
potentials $\wh{F}=\lfloor f(y)\rfloor-\lfloor f(y)\rfloor$ and $\wt{F}=\langle
f(x)\rangle -\langle f(y)\rangle$, where $\langle t\rangle = t-\lfloor
t\rfloor$ is the fractional part of the real number $t$. We claim that
\begin{align}\label{EQ:FRACF}
\psi(|F+v|_+) - \varphi(|F+v|_-) = \psi(|\wh{F}+v|_+) - \varphi(|\wh{F}+v|_-)
+ \psi_S(\wt{F}) + \varphi_{S^c}(\wt{F}).
\end{align}
Indeed, note that for $(x,y)\in S$ we have $\wh{F}(x,y)+v(x,y)\ge 0$, and for
$(x,y)\notin S$ we have $\wh{F}(x,y)+v(x,y)\le 0$. Hence
\begin{align*}
\psi(|F+v|_+)= \psi_S(F+v) = \psi_S(\wh{F}+v) +  \psi_S(\wt{F})
&= \psi(|\wh{F}+v|_+) + \psi_S(\wt{F}).
\end{align*}
Similarly,
\begin{align*}
\varphi(|F+v|_-) =  \varphi(|\wh{F}+v|_-) - \varphi_{S^c}(\wt{F}).
\end{align*}
This proves \eqref{EQ:FRACF}.

Replacing $f$ by $f+a$ with any real constant $a$, the potential $F$ and the
set $S$ do not change, but the potentials $\wh{F}_a(x,y)=\lfloor
f(x)+a\rfloor-\lfloor f(y)+a\rfloor$ and $\wt{F}_a(x,y) =  \langle
f(x)+a\rangle -\langle f(y)+a\rangle$ do depend on $c$. We have
\begin{align*}
\psi(|F+v|_+) - \varphi(|F+v|_-) = \psi(|\wh{F}_a+v|_+)- \varphi(|\wh{F}_a+v|_-)
+\psi_S(\wt{F}_a)-\varphi_{S^c}(\wt{F}_a).
\end{align*}
Choosing $a$ randomly and uniformly from $[0,1]$, the expectation of the last
two terms is $0$, since $\E(\langle f(x)+a\rangle) = 1/2$ for any $x$, and so
$\E(\wt{F}_a(x,y)) =0$ for all $x$ and $y$. Thus
\begin{align*}
\psi(|F+v|_+) - \varphi(|F+v|_-) =\E\bigl(\psi(|\wh{F}_a+v|_+) - \varphi(|\wh{F}_a+v|_-)\bigr).
\end{align*}
This implies that there is an $a\in[0,1]$ for which
\begin{align*}
\psi(|F+v|_+) - \varphi(|F+v|_-) \ge \psi(|\wh{F}_a+v|_+) - \varphi(|\wh{F}_a+v|_-).
\end{align*}
So replacing $f$ by $\lfloor f+a\rfloor$, we get an integer valued potential
that violates condition \eqref{JJFB1} even more, which proves the Supplement.
\end{proof}

We can give a more combinatorial reformulation of Corollary
\ref{COR:ERG-FLOW-M}.

\begin{corollary}\label{COR:F-INT}
Given a measure $\psi\in\ca_+(\AA^2)$, there exists an ergodic circulation
$\eta$ such that $\eta\le \psi$ if and only if for every partition
$J=S_1\cup\dots\cup S_k$ into a finite number of Borel sets
\[
\sum_{1\le i\le j\le k} (j-i+1) \psi(S_j\times S_i) \ge 1.
\]
\end{corollary}

The (insufficient) cut condition discussed above corresponds to the case when
$k=2$.

\begin{proof}
Let $F(x,y)=f(y)-f(x)$ be a bounded integral valued potential. We may assume
that $f$ is integral valued and $1\le f\le k$ for some integer $k$. Then the
sets $S_i=\{x\in J:~f(x)=i\}$ $(i=1,\dots,k)$ form a partition of $J$. For
$x\in S_i$ and $y\in S_j$, we have
\[
|F(x,y)+1|_+ =
  \begin{cases}
    j-i+1, & \text{if $i\le j$}, \\
    0, & \text{otherwise}.
  \end{cases}
\]
Thus the condition in Corollary \ref{COR:ERG-FLOW-M} is equivalent to the
condition in Corollary \ref{COR:F-INT}.
\end{proof}

\subsection{Flows}

Let $\sigma,\tau\in\ca(\AA)$ be two measures with $\sigma(J)=\tau(J)$. We
consider $\sigma$ the ``supply'' and $\tau$, the ``demand''. We call a measure
$\varphi\in\ca_+(\AA^2)$ a {\it flow from $\sigma$ to $\tau$}, or briefly a
{\it $\sigma$-$\tau$ flow}, if $\varphi^1-\varphi^2=\sigma-\tau$. We may
assume, if convenient, that the supports of $\sigma$ and $\tau$ are disjoint,
since subtracting $\sigma\land\tau$ from both does not change their difference.
If this is the case, we call $\sigma(J)=\tau(J)$ the {\it value} of the flow.

Given two points $s,t\in J$, a measure $\varphi$ on $\AA^2$ such that
$\varphi^1-\varphi^2 = a(\delta_s-\delta_t)$ will be called an {\it $s$-$t$
flow of value $a$}. So $\varphi$ is a flow serving supply $a\delta_s$ and
demand $a\delta_t$.

Note that every measure $\varphi\in\ca_+(\AA^2)$ is a flow from $\varphi^1$ to
$\varphi^2$, and also a flow from $\varphi^1\setminus\varphi^2$ to
$\varphi^2\setminus\varphi^1$. But we are usually interested in starting with
the supply and the demand, and constructing appropriate flows. We may require
$\varphi$ to be acyclic, since subtracting a circulation does not change
$\varphi^1-\varphi^2$.

As before, we may also be given a nonnegative measure $\psi$ on $\AA^2$ (the
``edge capacity''). We call a flow $\varphi$ {\it feasible}, if
$\varphi\le\psi$.

\subsubsection{Max-Flow-Min-Cut and Supply-Demand}

These fundamental theorems follow from the results on circulations by the same
tricks as in the finite case.

\begin{theorem}[Max-Flow-Min-Cut]\label{THM:MFMC-M}
Given a capacity measure $\psi\in\ca_+(\AA^2)$ and two points $s,t\in J$, there
is a feasible $s$-$t$ flow of value $1$ if and only if $\psi(A\times A^c)\ge 1$
for every $A\in\AA$ with $s\in A$ and $t\notin A$.
\end{theorem}

\begin{proof}
For every feasible flow $\phi\le\psi$ of value $1$, the measure
$\phi+\delta_{ts}$ is a circulation such that $\delta_{st}\le
\phi+\delta_{st}\le \psi+\delta_{st}$. Conversely, for every circulation
$\alpha$ with $\delta_{ts}\le \alpha \le \psi+\delta_{st}$, the measure
$\alpha-\delta_{ts}$ is a feasible $s$-$t$ flow of value $1$. The conditions in
Theorem \ref{THM:HOFF-CIRC-M} on the existence of such a circulation are
trivial except for the second condition when $s\in A$ and $t\notin A$, which
gives the condition in the theorem.
\end{proof}

The more general Supply-Demand Theorem can be stated as follows.

\begin{theorem}\label{THM:SUP-DEM-MEAS}
Let $\psi\in\ca_+(\AA^2)$, and let $\sigma,\tau\in\ca_+(\AA)$ with
$\sigma(J)=\tau(J)$. Then there is a feasible $\sigma$-$\tau$ flow if and only
if $\psi(S\times S^c)\ge \sigma(S)-\tau(S)$ for every $S\in\AA$.
\end{theorem}

\begin{proof}
We may assume that $\sigma(J)=\tau(J)=1$. Add two new points $s$ and $t$ to
$J$, and extend $\AA$ to a sigma-algebra $\AA'$ on $J'=J\cup\{s,t\}$ generated
by $\AA$, $\{s\}$ and $\{t\}$. Define a new capacity measure $\psi'$ by
\[
\psi'(X)=
  \begin{cases}
    \psi(X), & \text{if $X\subseteq J\times J$}, \\
    \sigma(Y), & \text{if $X= \{s\}\times Y$ with $Y\subseteq J$}, \\
    \tau(Y), & \text{if $X= Y\times \{t\}$ with $Y\subseteq J$},\\
    0, & \text{if $X\subseteq (\{t\}\times J) \cup (J\times \{s\}) \cup \{st,ts\}$},
  \end{cases}
\]
and extend it to all Borel sets by additivity. For every feasible
$\sigma$-$\tau$ flow $\phi$ on $(J,\AA)$, the measure $\phi+\psi'_{\{s\}\times
J}+\psi'_{J\times\{t\}}$ is a feasible $s$-$t$ flow of value $1$. Conversely,
for every feasible $s$-$t$ flow of value $1$, its restriction to the original
space $(J,\AA)$ is a feasible $\sigma$-$\tau$ flow. Applying the condition in
the Max-Flow-Min-Cut Theorem completes the proof.
\end{proof}

The measure-theoretic Max-Flow-Min-Cut Theorem is closely related to a result
of Laczkovich \cite{Lacz}, who works in the function setting. He also states an
integrality result, which is in a sense dual to our integrality result in
Section \ref{SSEC:INTEGER}.

A condition for the minimum cost of a feasible $\sigma$-$\tau$ flow of a given
value can be derived from Theorem \ref{THM:MAX-CIRC} using the same kind of
constructions as in the proof above. This gives the following result.

\begin{theorem}\label{THM:MIN-COST-FLOW}
Given a bounded measurable ``cost'' function $v:~J\times J\to \R_+$, a
``capacity'' measure $\psi\in\ca_+(\AA^2)$ and ``supply-demand'' measures
$\sigma,\tau\in\ca_+(\AA)$ with $\sigma(J)=\tau(J)$, there is a feasible
$\sigma$-$\tau$ flow $\varphi$ with $\varphi(v)=1$ if and only if
\begin{equation}\label{EQ:JJFVP2}
\psi(|f(y)-f(x)+bv(x,y)|_+) \ge \tau(f)-\sigma(f)+b
\end{equation}
for every bounded measurable function $f:~J\to\R$ and
$b\in\{-1,0,1\}$.\proofend
\end{theorem}

\subsubsection{Transshipment}

An optimization problem closely related to flows is the {\it transshipment
problem}. In its simplest measure-theoretic version, we are given two measures
$\alpha,\beta\in\ca(\AA)$ with $\alpha(J)=\beta(J)$. An {\it $\alpha$-$\beta$
transshipment} is a measure $\mu\in\ca_+(\AA\times\AA)$ coupling $\alpha$ and
$\beta$; in other words, $\mu^1=\alpha$ and $\mu^2=\beta$. Note the difference
with the notion of an $\alpha$-$\beta$ flow: there only the difference
$\mu^1-\mu^2$ is prescribed. In transhipment problems, one can think of
$J\times J$ as the edge set of a (complete) bipartite graph whose color classes
are the two copies of $J$. This observation can be used to derive the following
result from the Supply-Demand Theorem \ref{THM:SUP-DEM-MEAS}:

\begin{theorem}\label{THM:FEAS-TRANSSHIP}
Let $(J,\AA)$ be a standard Borel space, and $\alpha,\beta\in\ca_+(\AA)$ with
$\alpha(J)=\beta(J)$. Let $\psi\in\ca_+(\AA\times\AA)$. Then there exists an
$\alpha$-$\beta$ transshipment $\mu$ with $\mu\le\psi$ if and only if
\[
\psi(S\times T)\ge \alpha(S)+\beta(T)-\alpha(J)
\]
for every $S,T\in\AA$.\proofend
\end{theorem}

Suppose that every edge $(x,y)\in J\times J$ has a given cost $c(x,y)\ge 0$. We
want to find a transshipment minimizing the cost $\mu(c)$. We note that the
minimum is attained by Lemma \ref{LEM:COUPLE-CONV}.

\begin{theorem}\label{THM:MINCOST-TRANSSHIP}
Let $(J,\AA)$ be a standard Borel space, and $\alpha,\beta\in\ca_+(\AA)$ with
$\alpha(J)=\beta(J)$. Let $c:~J\times J\to\R_+$ be a bounded measurable
function. Then the minimum cost of an $\alpha$-$\beta$ transshipment is
$\sup_{g,h} \alpha(g)+\beta(h)$, where $g$ and $h$ range over all bounded
measurable functions $J\to\R$ satisfying $g(x)+h(y)\le c(x,y)$ for all $x,y\in
J$.
\end{theorem}

The proof follows by an easy reduction to Theorem \ref{THM:MIN-COST-FLOW}.

As a third variation on the Transshipment Problem, we ask for a transhipment
supported on a specified set $E$ of pairs. The following result is a slight
generalization of a theorem of Strassen \cite{Str}, and essentially equivalent
to Proposition 3.8 of Kellerer \cite{Kell}. See also \cite{Feld}. It is also a
rather straightforward generalization of Theorem 2.5.2 in \cite{Mbook}. The
result could also be considered as a limiting case of Theorem
\ref{THM:MINCOST-TRANSSHIP}, using the capacity ``measure'' with infinite
values on $E$.

\begin{prop}\label{PROP:STRASSEN}
Let $(J,\AA)$ be a standard Borel space, and $\alpha,\beta\in\ca_+(\AA)$ with
$\alpha(J)=\beta(J)=1$. Let $E\in \AA\times \AA$ be a Borel set such that
$J\times J\setminus E$ is the union of a countable number of product sets
$A\times B$ $(A,B\in\AA)$. Then there exists an $\alpha$-$\beta$ transshipment
$\mu$ concentrated on $E$ if and only if $\alpha(S)+\beta(T)\le1$ for any two
sets $S,T\in\AA$ with $S\times T\cap E=\emptyset$.
\end{prop}

\begin{remark}\label{REM:BIRK-VN}
In the finite case, the fundamental Birkhoff--von Neumann Theorem describes the
extreme points of the convex polytope formed by doubly stochastic matrices:
these are exactly the permutation matrices, or in the language of bipartite
graphs, perfect matchings. One generalization of this problem to the measurable
case is to consider the set of coupling measures between two copies of a
probability space $(J,\AA,\pi)$, forming a convex set in $\ca_+(\AA^2)$. What
are the extreme points (coupling measures) of this convex set? Unfortunately,
these extreme points seem to be too complex for an explicit description. See
\cite{Los} for several examples.
\end{remark}

\subsubsection{Path decomposition}

In finite graph theory, it is often useful to decompose an $s$-$t$ flow into a
convex combination of flows along single paths from $s$ to $t$ and circulations
along cycles. We will also need a generalization of this construction to
measurable spaces.

Let $K=J\cup J^2\cup J^3\cup\dots$ be the set of all finite nonempty sequences
of points of $J$; we also call these {\it walks}. The set $K$ is endowed with
the sigma-algebra $\BB=\AA\oplus\AA^2 \oplus\dots$. Let $K(s,t)$ be the subset
of $K$ consisting of walks starting at $s$ and ending at $t$ ($s,t\in J$); such
a walk is called an {\it $s$-$t$ walk}.

Let $\tau\in\ca_+(\BB)$. For $Q=(u^0,u^1,\dots,u^m)\in K$, let
$Q'=(u^0,\dots,u^{m-1})$, $V(Q)=\{u^0,\dots,u^m\}$,
$E(Q)=\{u^0u^1,u^1u^2,\dots,u^{m-1}u^m\}$, and $Z(Q)=\{u^0,u^m\}$. Define
\begin{align*}
V(\tau)(X) &= \int\limits_K |V(Q')\cap X|\,d\tau(Q) \qquad (X\in\AA)\\
E(\tau)(Y) &= \int\limits_K |E(Q)\cap Y|\,d\tau(Q)  \qquad (Y\in\AA^2),\\
Z(\tau)(Y) &= \int\limits_K |Z(Q)\cap Y|\,d\tau(Q)  \qquad (Y\in\AA^2).
\end{align*}
Then $V(\tau)$ is a measure on $\AA$, and $E(\tau)$ and $Z(\tau)$ are measures
on $\AA^2$. The measure $Z(\tau)$ is finite, but $V(\tau)$ and $E(\tau)$ may
have infinite values as for now. If $\tau$ is a probability measure, then
walking along a randomly chosen walk from distribution $\tau$, $V(\tau)(X)$ is
the expected number of times we exit a point in $X$ (so the starting point
counts, but the last point does not), and $E(\tau)(Y)$ is the expected number
of times we traverse an edge in $Y$. Mapping each walk $W\in K$ to its first
point, and pushing $\tau$ forward by this map, we get the measure
$Z(\tau)^1\in\ca(\AA)$. The measure $Z(\tau)^2$ is characterized analogously by
mapping each walk to its last point. It is easy to see that $E(\tau)$ is a flow
from $Z(\tau)^1$ to $Z(\tau)^2$.

\begin{theorem}\label{THM:DEMAND-PATH}
For every acyclic measure $\varphi\in\ca_+(\AA^2)$ there is a finite measure
$\tau\in\ca_+(\BB)$ for which $E(\tau)=\varphi$.
\end{theorem}

We need a simple (folklore) fact about Markov chains.

\begin{lemma}\label{LEM:RAND-HIT}
Let $\Gb$ be an indecomposable Markov space, and let $S\in\AA$ have $\pi(S)>0$.
Then for $\pi$-almost-all starting points $x$, a random walk started at $x$
hits $S$ almost surely.
\end{lemma}

\begin{proof}
Let $R$ be the set of starting points $x\in J$ for which the random walk
starting at $x$ avoids $S$ with positive probability, and suppose that
$\pi(R)>0$. Since clearly $R\cap S=\emptyset$, we also have $\pi(R)<1$. Hence
$\eta(R^c\times R)>0$ by indecomposability, and so there must be a point $x\in
R^c$ with $P_x(R)>0$. But this means that starting at $x$, the walk moves to
$R$ with positive probability, and then avoids $S$ with positive probability,
so we would have $x\in R$, a contradiction.
\end{proof}

\begin{proof*}{Theorem \ref{THM:DEMAND-PATH}}
We start with the special case when $\varphi$ is an $s$-$t$ flow for $s,t\in
J$; we may scale it to have value $1$. Just as in the proof of Theorem
\ref{THM:MFMC-M}, we see that the measure $\alpha=\varphi + \delta_{ts}$ is a
nonnegative circulation on $\AA^2$. Let $a=\alpha(J\times J) = \varphi(J\times
J)+1$, then $\eta=\alpha/a$ is the ergodic circulation of a Markov space. The
stationary distribution of this Markov space is $\pi=\alpha^1/a=\alpha^2/a$,
and
\begin{equation}\label{EQ:PHIJ}
\varphi^1 = a\pi - \delta_t.
\end{equation}
It is easy to see that $\varphi(\{(s,s)\})=0$, since
$\xi=\varphi(\{(s,s)\})\delta_{\{(s,s)\}}$ is a nonnegative circulation such
that $\xi\le\varphi$, and since $\varphi$ is acyclic, we must have $\xi=0$.

\begin{claim}\label{CLAIM:IRRED}
The Markov space $(\AA,\eta)$ is indecomposable.
\end{claim}

Indeed, suppose that there is a set $A\in\AA$ with $0<\pi(A)<1$ and
$\eta(A\times A^c)=\eta(A^c\times A)=0$. Clearly $s$ and $t$ either both belong
to $A$ or both belong to $A^c$; we may assume that $s,t\in A^c$. Then
$\varphi_{A\times A}$ is a circulation, and $\varphi=(\varphi-\varphi_{A\times
A})+\varphi_{A\times A}$ is a decomposition showing that $\varphi$ is not
acyclic, contrary to the hypothesis.

To specify a probability distribution on $s$-$t$ walks, we describe how to
generate a random $s$-$t$ walk: Start a random walk at $s$, and follow it until
you hit $t$ or return to $s$, whichever comes first. This happens almost surely
by Lemma \ref{LEM:RAND-HIT}: the distribution $\delta_s$ is absolutely
continuous with respect to $\pi$, and $\pi(t)>0$. This gives a probability
distribution $\tau$ on the set $K(s,\{s,t\})$ of walks from $s$ to $\{s,t\}$.

Let us stop the walk after $k$ steps, or when it hits $t$, or when it returns
to $s$, whichever comes first. This gives us a distribution $\tau_k$ over walks
starting at $s$ of length at most $k$. We claim that this distribution
satisfies the following identity for every $X\subseteq J\setminus\{s,t\}$:
\begin{equation}\label{EQ:SIGMA}
V(\tau_n)(X)= \int\limits_{J\setminus\{s,t\}}
P_u(X)\,dV(\tau_{n-1})(u).
\end{equation}
Indeed, let $\sigma_k(X)$ $(X\in\AA)$ be the probability that starting at $s$,
we walk $k$ steps without hitting $t$ or returning to $s$, and after $k$ steps
we are in $X$. It is clear that $\sigma_0=\delta_s$. It is also easy to see
that for $n\ge1$, we have $V(\tau_n)=\sigma_0+\sigma_1+\dots+\sigma_{n-1}$, and
for $X\subseteq J\setminus\{s,t\}$,
\begin{equation}\label{EQ:S-REC}
\sigma_n(X)= \int\limits_{J\setminus \{t\}} P_u(X)\,d\sigma_{n-1}(u).
\end{equation}
Thus
\[
V(\tau_n)(X) = \sum_{k=1}^{n-1} \sigma_k(X)
= \sum_{k=1}^{n-1}  \int\limits_{J\setminus\{t\}} P_u(X)\,d\sigma_{k-1}(u)
= \int\limits_{J\setminus\{t\}} P_u(X)\,dV(\tau_{n-1})(u).
\]
This proves \eqref{EQ:SIGMA}.

Next we show that
\begin{equation}\label{EQ:TAUBAR-N}
V(\tau_n) \le\varphi^1\qquad(n\ge 1).
\end{equation}
We prove the inequality by induction on $n$. For $n=1$ it is obvious. Let $n\ge
2$. If $s,t\notin X$, then $\sigma_0(X)=0$, and so using \eqref{EQ:SIGMA} and
\eqref{EQ:PHIJ},
\begin{align*}
V(\tau_n)(X) &= \int\limits_{J\setminus\{t\}} P_u(X)\,dV(\tau_{n-1})(u)\\
&\le \int\limits_{J\setminus\{t\}} P_u(X)\,d\varphi^1(u)\le
a\int\limits_{J\setminus\{t\}} P_u(X)\,d\pi(u) \\
&\le a\int\limits_J P_u(X)\,d\pi(u) = a\pi(X)= \varphi^1(X).
\end{align*}
If $t\in X$ but $s\notin X$, then
\[
V(\tau_n)(X) = V(\tau_n)(X\setminus\{t\})
\le \varphi^1(X\setminus\{t\}) \le \varphi^1(X).
\]
If $s\in X$, then (using that every random walk we constructed exits $s$ only
once)
\[
V(\tau_n)(X) = 1+V(\tau_n)(X\setminus\{s\})
\le 1+\varphi^1(X\setminus\{s\})\le\varphi^1(X).
\]

Next, we consider $E(\tau)$, which is an $s$-$t$ flow by the discussion before
the theorem. It follows easily that
\begin{equation}\label{EQ:WTTAU-PHI}
E(\tau_n) \le\varphi \qquad (n\ge 1).
\end{equation}
Indeed, for $A,B\in\AA$,
\[
E(\tau_n)(A\times B) = \int\limits_A P_u(B)\,dV(\tau_n)(u)
\le\int\limits_A P_u(B)\,d\varphi^2(u)=\varphi(A\times B).
\]
This implies that $E(\tau_n)(X)\le\varphi(X)$ for every $X\in\AA^2$, proving
\eqref{EQ:WTTAU-PHI}.

\begin{claim}\label{CLAIM:TAU-CONV}
$V(\tau_n)\to V(\tau)$ in total variation distance.
\end{claim}

Since clearly $V(\tau_n)\le V(\tau)$, we have
$d_\tv(V(\tau_n),V(\tau))=V(\tau)(J)-V(\tau_n)(J)$. Let $p_n$ be the
probability that a random walk started at $s$ first hits $\{s,t\}$ in exactly
$n$ steps. Then
\[
V(\tau)(J)=\sum_{k=1}^\infty p_k\,k, \qquad\text{and}\qquad
V(\tau_n)(J)=\sum_{k=1}^n p_k\,k.
\]
By \eqref{EQ:TAUBAR-N}, $V(\tau_n)(J)\le \varphi^1(J)<\infty$, and hence the
series representing $\tau$ is convergent. This proves the claim.

\begin{claim}\label{CLAIM:S-NON}
The probability that a random walk started at $s$ returns to $s$ before hitting
$t$ is zero. So $\tau$ can be considered as a probability distribution on walks
from $s$ to $t$.
\end{claim}

Indeed, we can split $K(s,\{s,t\})=K(s,s)\cup K(s,t)$. Define $\rho
=\tau_{K(s,s)}$. Then $E(\rho)\le E(\tau)\le\varphi$ and it is easy to see that
$E(\rho)$ is a circulation. Since $\varphi$ is acyclic, we must have $\rho=0$,
and so $\tau(K(s,s))=0$.

Inequalities \eqref{EQ:TAUBAR-N}, \eqref{EQ:WTTAU-PHI} and Claim
\ref{CLAIM:TAU-CONV} imply that $V(\tau)\le\varphi^1$ and $E(\tau) \le
\varphi$. To complete the proof, consider the measure $\varphi-E(\tau)$. This
is a nonnegative circulation, and since $\varphi$ is acyclic, it follows that
$\varphi-E(\tau)=0$. This proves the theorem for $s$-$t$ flows.

The general case can be reduced to the special case of an $s$-$t$ flow by the
following construction, similar to that used in the proof of Theorem
\ref{THM:SUP-DEM-MEAS}. Let $\varphi\in\ca_+(\AA^2)$ be an acyclic measure, let
$\sigma=\varphi^1\setminus\varphi^2$ and $\tau=\varphi^2\setminus\varphi^1$, so
that $\varphi$ is an acyclic $\sigma$-$\tau$ flow. Create two now points $s$
and $t$, extend $\AA$ to a sigma-algebra $\AA'$ on $J'=J\cup\{s,t\}$ generated
by $\AA$, $\{s\}$ and $\{t\}$, and extend the measure $\varphi$ to
$\varphi'\in\ca(\AA'\times\AA')$ by
\[
\varphi'(X)=
  \begin{cases}
    \varphi(X), & \text{if $X\subseteq J\times J$}, \\
    \sigma(Y), & \text{if $X= \{s\}\times Y$ with $Y\subseteq J$}, \\
    \tau(Y), & \text{if $X= Y\times \{t\}$ with $Y\subseteq J$},\\
    0, & \text{if $X\subseteq (\{t\}\times J) \cup (J\times \{s\}) \cup \{st,ts\}$},
  \end{cases}
\]
It is easy to check that $\varphi'$ is an acyclic $s$-$t$ flow. Using the
theorem for the special case of this $s$-$t$ flow, we get a measure $\tau$ on
$s$-$t$ paths, in which the trivial path $(s,t)$ has zero measure. So $\tau$
defines a measure on nontrivial $s$-$t$ paths, and since there is a natural
bijection with paths in $K$, we get a measure on $(K,\BB)$. It is easy to check
that this measure has the desired properties.
\end{proof*}

\begin{remark}\label{REM:CIRC-DECOMP}
Theorem \ref{THM:DEMAND-PATH} raises the question whether circulations have
analogous decompositions. In finite graph theory, a circulation can be
decomposed into a nonnegative linear combination of directed cycles. In the
infinite case, we have to consider, in addition, directed paths infinite in
both directions (see Example \ref{EXA:CA}); but even so, the decomposition is
not well understood.

Suppose that we have a nonnegative circulation $\eta\not=0$ on $\AA$. We may
assume (by scaling) that it is a probability measure, so it is the ergodic
circulation of a Markov space. From every point $u\in J$, we can start an
infinite random walk $(v^0=u, v^1,\dots)$, and also an infinite random walk
$(v^0=u, v^{-1},\dots)$ of the reverse chain. Choosing $u$ from $\pi$, this
gives us a probability distribution $\beta$ on rooted two-way infinite
(possibly periodic) sequences, i.e., on $J^{\Z}$. However, it seems to be
difficult to reconstruct the circulation $\alpha$ from $\beta$.
\end{remark}

\section{Multicommodity measures}

\subsection{Metrical linear functionals}

A bounded linear functional $\DD$ on $\ca(\AA^2)$ will be called {\it
metrical}, if it satisfies the following conditions:

\smallskip

(a) $\DD(\mu)=0$ for every measure $\mu\in\ca(\AA^2)$ concentrated on the
diagonal $\Delta=\{(x,x):~x\in J\}$;

\smallskip

(b) $\DD(\mu)=\DD(\mu^*)$ for every measure $\mu\in\ca(\AA^2)$;

\smallskip

(c) $\DD(\kappa^{12})+\DD(\kappa^{23})\ge\DD(\kappa^{13})$ for every measure
$\kappa\in\ca_+(\AA^3)$.

These conditions imply that $\DD$ is nonnegative on nonnegative measures.
Indeed, for a measure $\mu\in\ca_+(\AA^2)$ and an arbitrary probability
distribution $\gamma$ on $\AA$, define $\kappa=(\mu+\mu^*)\times\gamma$. Then
$\kappa^{12}=\mu+\mu^*$ and $\kappa^{13}=\kappa^{23}=
(\mu^1+\mu^2)\times\gamma$. Applying (c), we get that $\DD(\mu)+\DD(\mu^*)+
\DD((\mu^1+\mu^2)\times\kappa)\ge \DD((\mu^1+\mu^2)\times\kappa)$, and (b)
implies that $\DD(\mu)\ge 0$.

The name ``metrical'' refers to the fact that if $\DD$ is defined by a bounded
measurable pseudometric $r$ on $J$ as $\DD(\mu)=\mu(r)$, then conditions
(a)-(c) are satisfied. Conditions (a) and (b) are trivial, and condition (c)
also follows easily:
\[
\DD(\kappa^{12})+\DD(\kappa^{23})-\DD(\kappa^{13}) = \kappa^{12}(r) + \kappa^{23}(r) -\kappa^{13}(r)
= \kappa(r(y,z)+r(y,z)-r(x,z)) \ge 0.
\]

Can every metrical linear functional $\DD$ be represented as
$\DD(\varphi)=\varphi(g)$ with some pseudometric $g:~J^2\to\R_+$? I expect that
the answer is negative, but perhaps the following is true:

\begin{conjecture}\label{CONJ:METRICAL}
For every metrical linear functional $\DD$ on $\ca(\AA^2)$ and every
$\psi\in\ca_+(\AA^2)$ there is a pseudometric $g:~J^2\to\R_+$ such that
$\DD(\varphi)=\varphi(g)$ for all measures $\varphi\ll\psi$.
\end{conjecture}

The conjecture can proved in several special cases, in particular, for measures
$\psi$ defined by graphons and graphings (details will be published elsewhere).

We need a lemma relating metrical functionals and flows. Informally, the lemma
expresses that in a flow, every particle must travel at least as much as the
distance between its starting and ending points.

\begin{lemma}\label{LEM:SHORTCUT}
Let $\DD$ be a metrical linear functional on $\ca(\AA^2)$, and let
$\tau\in\ca_+(\BB)$. Then $\DD(E(\tau))\ge\DD(Z(\tau))$.
\end{lemma}

\begin{proof}
Let $\tau_k$ denote the measure $\tau$ restricted to sequences in $\BB$ of
length $k$ $(k\ge 1)$. For $0\le i_1<i_2<\dots<i_m<k$, the measure
$\tau_k^{i_1\dots i_m}$ is the marginal of $\tau_k$ on $\{i_1,\dots,
i_m\}\subseteq\{1,\dots,k\}$. For $i\le j$, let $[i,j]=\{i,i+1,\dots,j\}$. Then
$Z(\tau) = \sum_{k\ge0} \tau_k^{0,k-1}$.

We claim that
\begin{equation}\label{EQ:PATH1}
\DD(E(\tau_k^{[i,j]}))\ge\DD(E(\tau_k^{ij}))\qquad(0\le i<j<k).
\end{equation}

We use induction on $j-i$. For $j-i=1$ the assertion is trivial. Let $j-i>1$,
and choose $r$ with $i<r<j$. Then
\[
E(\tau_k^{irj})^{23} = E(\tau_k^{rj}),\quad E(\tau_k^{ij})^{13}
= E(\tau_k^{ij}),\quad E(\tau_k^{irj})^{12} = E(\tau_k^{ir}).
\]
Using that $\DD$ is metrical, this implies that
\[
\DD(E(\tau_k^{ir}))+\DD(E(\tau_k^{rj})) \ge \DD(E(\tau_k^{ij})).
\]
By induction, we know that $\DD(E(\tau_k^{[i,r]}))\ge\DD(E(\tau_k^{ir}))$ and
$\DD(E(\tau_k^{[r,j]}))\ge\DD(E(\tau_k^{rj}))$. Using that
$E(\tau_k^{[i,r]})+E(\tau_k^{[r,j]})=E(\tau_k^{[i,j]})$, we get
\[
\DD(E(\tau_k^{[i,j]})) = \DD((\tau_k^{[i,r]}))+\DD(E(\tau_k^{[r,j]}))
\ge \DD(E(\tau_k^{ir}))+\DD(E(\tau_k^{rj})) \ge \DD(E(\tau_k^{ij})).
\]
This proves the Claim. In particular, we have
\begin{equation}\label{EQ:PATH4}
\DD(E(\tau_k))= \DD(E(\tau_k^{[0,k-1]})) \ge\DD(E(\tau_k^{0,k-1})) = \DD(Z(\tau_k)).
\end{equation}
Thus
\[
\DD(\tau) = \sum_{k=1}^\infty \DD(E(\tau_k))
\ge\sum_{k=0}^\infty \DD(Z(\tau_k)) =\DD(Z(\tau)).\qedhere
\]
\end{proof}

\subsection{Multicommodity flows}

A {\it multicommodity flow} on a Borel space $\AA$ consists of a symmetric
measure $\sigma\in\ca_+(\AA^2)$, and of a family of $s$-$t$ flows
$\varphi_{st}$ of value $1$, one for each pair $(s,t)\in J\times J$. We require
that $\varphi_{st}(U)$ is measurable as a function of $(s,t)\in J\times J$ for
every $U\in\AA^2$.

Since we are going to put only symmetric upper bounds (capacity constraints) on
the sum of these flows, we may also require that each $\varphi_{st}$ is
acyclic. A further requirement we can impose is that
$\varphi_{ts}=\varphi_{st}^*$ (replacing each $\varphi_{st}$ by
$(\varphi_{st}+\varphi_{ts}^*)/2$).

Such a multicommodity flow $F=(\sigma;~f_{st}:~st \in W)$ defines symmetric
measure (the {\it total load}) by
\[
\varphi_F(S)=\int\limits_{J\times J} \varphi_{xy}(S)\,d\sigma(x,y)\qquad (S\in\AA^2).
\]
A trivial multicommodity flow is defined by $f_{st}=\delta_{st}$ for any
$\sigma$. The total load of this trivial multicommodity flow is $\sigma$.

If we are also given a symmetric ``capacity'' measure $\psi\in\ca_+(\AA^2)$,
then we say that the multicommodity flow $F=(\sigma;~\varphi_{st})$ is {\it
feasible}, if $\varphi_F\le\psi$. Our question is: Given $\psi$ and $\sigma$,
does there exist a feasible multicommodity flow? Our goal is to generalize the
Multicommodity Flow Theorem.

To state our main result in this section, we need to relax the capacity
constraint $\varphi_F\le\psi$, and define the {\it overload over $\psi$} as
$\|\varphi_F\setminus\psi\|$. In other words, this overload is less than $\eps$
if there is a measure $\psi'\in\ca_+(\AA^2)$ such that $\|\psi-\psi'\|<\eps$
and $F$ is feasible with respect to $\psi'$.

\begin{theorem}[Multicommodity Flow Theorem for Measures]\label{THM:MULTI-COMM-M}
Let $\sigma$ and $\psi$ be symmetric measures on $\AA^2$. There is a feasible
multicommodity flow for demands $\sigma$ with arbitrarily small overload over
$\psi$ if and only if $\DD(\sigma)\le\DD(\psi)$ for every metrical linear
functional $\DD$ on $\ca_+(\AA^2)$.
\end{theorem}

I don't know whether allowing an arbitrarily small overload is needed (probably
so). If Conjecture \ref{CONJ:METRICAL} above is true, then the condition
$\DD(\sigma)\le\DD(\psi)$ could be replaced by the more explicit condition that
$\sigma(d)\le\psi(d)$ for every bounded Borel pseudometric $d$ on $J$.

A {\it cut-metric} is perhaps the simplest nontrivial pseudometric, defined as
$d(x,y)=\one_{A\times A^c}+\one_{A^c\times A}$. For cut-metrics, the condition
$\DD(\sigma)\le\DD(\psi)$ in the theorem gives that $\sigma(A\times A^c)\le
\psi(A\times A^c)$. If the demand measure $\sigma$ is concentrated on a single
pair $\{s,t\}$ of nodes (more exactly, on the two orderings of an unordered
pair), then we obtain Theorem \ref{THM:MULTI-COMM-M} (at least in the case of
symmetric capacities). But in general, it does not suffice to apply the
condition to cut-metrics only, even in the finite case.

\subsubsection{Formulation as a single measure}

We want to formulate the multicommodity flow problem in terms of a single
measure; unfortunately, we have to go up to $\AA^4$. If $\Phi\in\ca(\AA^4)$,
then we use the notation
\[
\Phi^*(T\times U)=\Phi(T^*\times U),\quad \Phi^{**}(T\times U)=\Phi(T^*\times U^*),
\quad \Phi^{\circ *}(T\times U)=\Phi(T\times U^*).
\]

Every multicommodity flow $(\sigma;~\varphi_{st}:~s,t\in J)$ defines a {\it
load measure} $\Phi$ on $\AA^4 = \AA^2\times\AA^2$ by
\[
\Phi(T\times U) = \int\limits_U \varphi_{st}(T)\,d\sigma(s,t).
\]
This number expresses how much load the subset of demands $U$ puts on the edges
in $T$. For the trivial solution $\varphi_{st}=\delta_{st}$ (sending the stuff
directly from $s$ to $t$) we get
\[
\int\limits_U \delta_{xy}(T)\,d\sigma(x,y) = \sigma(T\cap U).
\]
Sometimes it will be convenient to consider the right hand side as a measure
$\sigma_\Delta(T\times U)=\sigma(T\cap U)$ defined on $\AA^4$. Of course, this
trivial solution is not feasible in general.

We can express the multicommodity flow problem in terms of this single measure
$\Phi$. The condition that $\varphi_{st}^*=\varphi_{ts}$ can be expressed as
$\Phi(T\times U)=\Phi(T^*\times U^*)$, or more compactly,
\begin{equation}\label{EQ:L-SYM1}
\Phi^{**}=\Phi.
\end{equation}
The fact that $\varphi_{st}-\delta_{st}$ is a circulation implies that
\[
\varphi_{st}^1(A) - \varphi_{st}^2(A)= \delta_{st}^1(A) -\delta_{st}^2(A) = \delta_s(A)-\delta_t(A) \qquad(A\in\AA).
\]
Integrating over $U\in\AA^2$ with respect to $\sigma$, we get that
\begin{equation}\label{EQ:FLO-M-LOAD2}
\Phi^{134}-\Phi^{234} = \overline{\sigma},
\end{equation}
where $\overline{\sigma}(A\times U) = \sigma((A\times J)\cap U) -
\sigma((J\times A)\cap U)$.

Finally, the feasibility conditions mean that $\Phi\ge 0$ and $\Phi(A\times
J\times J)\le \psi(A)$, which, using our notation, can be expressed as
\begin{equation}\label{EQ:FLO-M-FEAS}
\Phi\ge 0, \qquad \Phi^{12} \le \psi.
\end{equation}

Our next observation is that we can forget about condition \eqref{EQ:L-SYM1}.
Indeed, suppose that $\Phi\in\AA^4$ satisfies \eqref{EQ:FLO-M-LOAD2} and
\eqref{EQ:FLO-M-FEAS}. Then the measure $\Phi^{**}$ also satisfies these
conditions, and the symmetrized measure $\frac12(\Phi+\Phi^{**})$ satisfies
these equations and, in addition, \eqref{EQ:L-SYM1} as well.

Conversely, we show that every measure $\Phi$ satisfying \eqref{EQ:FLO-M-LOAD2}
and \eqref{EQ:FLO-M-FEAS} yields a feasible multicommodity flow.

We may assume that $\Phi^{34}\ll \sigma$. Suppose this does not hold, then let
$S\in\AA^2$ be a set with $\sigma(S)=0$ and $\Phi^{34}(S)$ maximum (such a set
clearly exists). Define $\Phi_1=\Phi_{J^2\times (J^2\setminus S)}$ and
$\Phi_2=\Phi_{J^2\times S}$, then $\Phi=\Phi_1+\Phi_2$. We claim that
$\Phi_1\ll \sigma$. Indeed, for $X\subseteq J^2$ with $\sigma(X)=0$ we have
$\sigma(X\cup S)=0$, hence $\Phi^{34}(X\cup S)\le \Phi^{34}(S)$, which implies
that $\Phi^{34}_1(X)=\Phi^{34}(X\setminus S) =\Phi^{34}(X\setminus S)=0$.

Furthermore, $\Phi_1$ satisfies \eqref{EQ:FLO-M-LOAD2} and
\eqref{EQ:FLO-M-FEAS}. The second of these is trivial. For the first,
\begin{align*}
\Phi_1^{134}(A\times U)&-\Phi_1^{234}(A\times U) = \Phi_1(A\times J\times U)-\Phi_1(J\times A\times U)\\
&=\Phi(A\times J\times (U\setminus S))-\Phi(J\times A\times (U\setminus S))\\
&=\overline{\sigma}(A\times (U\setminus S)) = \sigma((A\times J)\cap (U\setminus S)) - \sigma((J\times A)\cap(U\setminus S))\\
& = \sigma((A\times J)\cap U) - \sigma((J\times A)\cap U) = \overline{\sigma}(A\times U)
\end{align*}
(we have used that $\sigma(S)=0$). Replacing $\Phi$ by $\Phi_1$ we get a
solution of \eqref{EQ:FLO-M-LOAD2} and \eqref{EQ:FLO-M-FEAS} such that
$\Phi^{34}\ll\sigma$. Thus the Radon--Nikodym derivative $f=d\Phi^{34}/d\sigma$
exists.

The Disintegration Theorem \ref{PROP:DISINT} implies that there is a family
$(\theta_{st}:~s,t\in J)$ of measures on $\AA^2$ such that $\theta_{st}(U)$ is
a measurable function of $(s,t)$ for every $U\in\AA^2$, and
\begin{equation}\label{EQ:PHISTU0}
\Phi(T\times U)=\int\limits_U \theta_{st}(T)\,d\Phi^{34}(s,t).
\end{equation}
for $T,U\in\AA^2$. Defining $\varphi_{st}= f(s,t)\cdot\theta_{st}$, equation
\eqref{EQ:PHISTU0} can be written as
\begin{equation}\label{EQ:PHISTU}
\Phi(T\times U)=\int\limits_U \varphi_{st}(T)\,d\sigma(s,t).
\end{equation}
Let $A\in\AA$ and $U\in\AA^2$, then
\begin{align*}
\int\limits_U(\varphi_{st}^1(A)&-\varphi_{st}^2(A))\,d\sigma(s,t) = \Phi^{134}(A\times U)-\Phi^{234}(A\times U)\\
&= \overline{\sigma}(A\times U) = \sigma((A\times J)\cap U)-\sigma((J\times A)\cap U)
=\int_U \one_{A\times J}-\one_{J\times A}\,d\sigma.
\end{align*}
This holds for every $U\in\AA$, so it follows that for all $A\in\AA$,
\begin{equation}\label{EQ:VFI-FLOW}
\varphi_{st}^1(A)-\varphi_{st}^2(A) = \one_{J\times A}(s,t)-\one_{A\times J}(s,t) = \delta_s(A)-\delta_t(A),
\end{equation}
holds for $\sigma$-almost all $(s,t)$. We need to argue that for
$\sigma$-almost all $(s,t)$, equation \eqref{EQ:VFI-FLOW} holds for all $A$.

Let $R_A$ denote the set of pairs $(s,t)$ for which \eqref{EQ:VFI-FLOW} does
not hold. Let $\{A_1,A_2,\dots\}$ be a countable set algebra generating $\AA$.
Then $R=\cup_i R_{A_i}$ has $\sigma(R)=0$ and if $(s,t)\notin R$, then
\[
\varphi_{st}^1(A_i)+\delta_t(A_i)=\varphi_{st}^2(A_i) + \delta_s(A_i).
\]
By the uniqueness of measure extension, this equality holds if we replace $A_i$
by any $A\in\AA$. This shows that $\varphi_{st}$ is an $s$-$t$ flow of value
$1$. Replacing $\varphi_{st}$ by $\delta_{st}$ for $(s,t)\in R$, we may assume
that $\varphi_{st}$ is an $s$-$t$ flow of value $1$ for every $s$ and $t$.

Equation \eqref{EQ:PHISTU} implies that
\[
\int_J\varphi_{st}(T)\,d\sigma(s,t) = \Phi(T\times J)\le \psi(T),
\]
so this multicommodity flow is feasible. If $\Phi$ violates the second
inequality in \eqref{EQ:FLO-M-FEAS} slightly, meaning that
$\|\Phi^{12}\setminus\psi\|=\eps>0$, then by a similar computation the
multicommodity flow we constructed has an overload of $\eps$.

To sum up, it suffices to find a measure $\Phi\in\ca_+(\AA^4)$ such that
$\Phi^{134}-\Phi^{234} =  \overline\sigma$ and $\|\Phi^{12} \setminus
\psi\|\le\eps$.

\subsubsection{Proof of the Multicommodity Flow Theorem}

I. {\it The ``only if'' direction.} Consider a multicommodity flow
$F=(\varphi^{uv}:~uv\in S)$, serving demand $\sigma$ and with overload over
$\psi$ less than $\eps$ ($\eps>0$). We may assume that $\sigma$ is a
probability distribution. By Theorem \ref{THM:DEMAND-PATH}, there is a
probability distribution $\kappa_{uv}$ on $u$-$v$ paths for every $uv\in S$
such that $E(\kappa_{uv})=\varphi_{uv}$. Let $\tau$ be the mixture of the
$\kappa_{uv}$ by $\sigma$; in other words, we generate a random path from
$\tau$ by selecting a random pair $uv$ from $\sigma$, and then select a random
path from $\kappa_{uv}$. Then $E(\tau) = \varphi_F$ and $Z(\tau)=\sigma$. By
the definition of overload, we have $\varphi_F\le\psi+\beta$, where
$\|\beta\|\le\eps$. By Lemma \ref{LEM:SHORTCUT},
\[
\DD(\sigma)=\DD(Z(\tau)) \le \DD(E(\tau))=\DD(\varphi_F)
\le \DD(\psi) + \DD(\beta) \le \DD(\psi)+\|\DD\|\eps.
\]
Since $\eps$ can be arbitrarily small, this proves that
$\DD(\sigma)\le\DD(\psi)$.

\medskip

II. {\it The ``if'' direction}. Consider the convex sets of measures
\begin{align*}\label{EQ:LOAD2}
\Hf_1&=\{\Phi\in\ca(\AA^4):~\Phi^{134}-\Phi^{234}=\overline{\sigma}\},\\
\Hf_2&=\ca_+(\AA^4),\\
\Hf_3&=\{\Phi\in\ca(\AA^4):~\Phi^{12}\le\psi\}.
\end{align*}
To make these sets open, let $\delta>0$, and consider the
$\delta$-neighborhoods
$\Hf_i^\delta=\{\mu\in\ca_+(\AA):~d_\tv(\mu,\Hf_i)<\delta\}$. Note that all
these sets are convex and invariant under the map $\Phi\mapsto \Phi^{**}$.

The main step in the proof is proving that
\begin{equation}\label{EQ:HXY}
\Hf_1^\delta\cap \Hf_2^\delta\cap \Hf_3^\delta\not=\emptyset.
\end{equation}
Suppose that this intersection is empty. The intersection of any two of these
sets is nonempty, so by Lemma \ref{LEM:HB-KN} there are bounded linear
functionals $\LL_1,\LL_2,\LL_3$ on $\ca(\AA^4)$ and real numbers $a_1,a_2,a_3$
such that $\LL_1+\LL_2+\LL_3=0$, $a_1+a_2+a_3=0$, and $\LL_i > a_i$ on
$\Hf_i^\delta$. Note that $0\in \Hf_2$ and $0\in \Hf_3$, which implies that
$a_2,a_3<0$, and hence $a_1>0$. Since the sets are invariant under the map
$\Phi\mapsto\Phi^{**}$, we may assume that the linear functionals $\LL_1,
\LL_2,\LL_3$ are invariant under this map as well.

These conditions have the following implications for the functionals $\LL_i$:

\smallskip

(a) The affine subspace $\Hf_1$ is not empty, since the trivial multicommodity
flow satisfies it. The condition that $\LL_1(\Phi)> a_1$ for $\Phi\in
\Hf_1^\delta$ implies that $\LL_1$ is constant on $\Hf_1$. Since $a_1>0$, this
constant is positive, and we may assume (by scaling the $\LL_i$ and the $a_i$)
that it is $1$. Then $a_1<1$. It follows that $\LL_1(\Phi)=0$ if
$\Phi^{134}=\Phi^{234}$.

We can apply Proposition \ref{PROP:TT-INV} to the linear operator
$\TT:~\Phi\mapsto\Phi^{134}-\Phi^{234}$ similarly as in the proof of Lemma
\ref{LEM:CIRC-DUAL}. We get a linear functional $\ZZ$ on $\ca(\AA^3)$ such that
\begin{equation}\label{EQ:LL-ZZ}
\LL_1(\Phi)=\ZZ(\Phi^{134}-\Phi^{234})\qquad(\Phi\in\ca(\AA^4)).
\end{equation}
Substituting the trivial multicommodity flow in \eqref{EQ:LL-ZZ}, we get that
$\ZZ(\overline{\sigma})=1$. It also follows that
\begin{equation}\label{EQ:LL1-ALT}
\LL_1(\Phi^*) = \ZZ((\Phi^*)^{134}-(\Phi^*)^{234}) = \ZZ(\Phi^{234}-\Phi^{134}) = - \LL_1(\Phi),
\end{equation}
and
\begin{equation}\label{EQ:LL1-ALT2}
\LL_1(\Phi^{\circ *}) = \LL_1((\Phi^{**})^*) = -\LL_1(\Phi^{**}) = - \LL_1(\Phi).
\end{equation}

\smallskip

(b) The condition that $\LL_2(\Phi)> a_2$ for $\Phi\in \Hf_2^\delta$ implies
that $\LL_2(\mu)\ge 0$ for $\mu\ge 0$, so $\LL_2$ is a nonnegative functional.

\smallskip

(c) The condition that $\LL_3(\Phi)> a_3$ for $\Phi\in \Hf_3^\delta$ implies
that $\LL_3(\mu)\ge 0$ whenever $\mu\in\ca(\AA^4)$ and $\mu^{12}\le 0$. This
implies that $\LL_3(\mu)=0$ whenever $\mu^{12}=0$. We can apply the Proposition
\ref{PROP:TT-INV} to the operator $\SS:~\varphi\mapsto\varphi^{12}$ similarly
as in (a); it is easy to see that the range of $\SS$ is the whole space
$\ca(\AA^2)$, so it is closed. We get a bounded linear functional $\RR$ on
$\ca(\AA^2)$ such that $\LL_3(\mu)=\RR(\mu^{12})$. It also follows that $-\RR$
is a nonnegative functional.

\smallskip

From $\LL_1+\LL_2+\LL_3=0$ we get that
\begin{equation}\label{EQ:RL}
\RR(\Phi^{12}) = -\LL_3(\Phi) = \LL_1(\Phi)+\LL_2(\Phi) \ge \LL_1(\Phi)
=\ZZ(\Phi^{134}-\Phi^{234}).
\end{equation}
for every $\Phi\in\ca_+(\AA^4)$. From the fact that $\psi\times\gamma\in \Hf_3$
for any probability measure $\gamma\in\ca(\AA^2)$, it follows that $\RR(\psi)<
- a_3 = a_1+a_2<1$.

By Lemma \ref{LEM:FUNCT-SUP}, there is a bounded linear functional $\QQ$ on
$\ca(\AA^2)$ such that
\[
\QQ(\mu)=\sup\{\LL_1(\Phi):~\Phi^{12}=\mu,~ \Phi\ge 0\}
= \sup\{\ZZ(\Phi^{134}-\Phi^{234}):~\Phi^{12}=\mu,~ \Phi\ge 0\}
\]
for all $\mu\ge 0$. Note that $\QQ(\mu)\le\RR(\mu)$ and
\begin{equation}\label{EQ:QZZPHI}
\QQ(\Phi^{12})\ge \ZZ(\Phi^{134}-\Phi^{234})
\end{equation}
for every $\Phi\ge0$. Also note that in the definition, the measure
$\Phi^{\circ *}$ also competes for the supremum, and since $\LL_1(\Phi^{\circ
*})=-\LL_1(\Phi)$, we can also write
\begin{equation}\label{EQ:Q-ABS}
\QQ(\mu)=\sup\{|\LL_1(\Phi)|:~\Phi^{12}=\mu,~ \Phi\ge 0\}\ge 0.
\end{equation}
We also have $\sigma_\Delta\ge 0$ and $(\sigma_\Delta)^{12}=\sigma$, and so
\begin{equation}\label{EQ:QQS1}
\QQ(\sigma)\ge \LL_1(\sigma_\Delta) = 1.
\end{equation}

\begin{claim}\label{CLAIM:R-MET}
The functional $\QQ$ is metrical.
\end{claim}

First, suppose that $\mu$ is concentrated on the diagonal of $\AA^2$. Then
every measure $\Phi\in\ca(\AA^4)$ with $\Phi^{12}=\mu$ is concentrated on the
set $\{(x,x,u,v):~x,u,v\in J\}$, and hence $\Phi^{134}=\Phi^{234}$, so
$\QQ(\mu)=0$.

Second, for every $\mu\ge 0$ we have
\begin{align*}
\QQ(\mu^*) &= \sup\{\LL_1(\Phi):~\Phi^{12}=\mu^*,~ \Phi\ge 0\}
= \sup\{\LL_1(\Phi):~(\Phi^{**})^{12}=\mu,~ \Phi\ge 0\}\\
&= \sup\{\LL_1(\Phi^{**}):~\Phi^{12}=\mu,~ \Phi\ge 0\} = \QQ(\mu).
\end{align*}

Third, let $\kappa\in\ca_+(\AA^3)$ and $\delta>0$. By the definition of $\QQ$,
there is a measure $\Phi\in\ca_+(\AA^4)$ such that
\[
\QQ(\kappa^{12})\le\LL_1(\Phi)+\delta,\qquad\text{and}\qquad \Phi^{12}=\kappa^{12}.
\]
Consider the space $\ca(\AA^{\{12345\}})$, where the space of $\kappa$ is
identified with $\ca(\AA^{\{125\}})$ (the space of $\Phi$ remains
$\ca(\AA^{\{1234\}})$. The equation $\Phi^{12}=\kappa^{12}$ implies that there
is a measure $\Gamma\in\ca_+(\AA^{\{12345\}})$ such that $\Gamma^{1234}=\Phi$
and $\Gamma^{125}=\kappa$. Using \eqref{EQ:LL-ZZ}, we get
\begin{align*}
\QQ(\kappa^{12})&=\QQ(\Phi^{12})\le\LL_1(\Phi)+\delta
= \ZZ(\Phi^{134}-\Phi^{234})+\delta = \ZZ(\Gamma^{134}-\Gamma^{234})+\delta\\
&=\ZZ(\Gamma^{134}-\Gamma^{345})+\ZZ(\Gamma^{345}-\Gamma^{234})+\delta
\end{align*}
Applying \eqref{EQ:QZZPHI} with $\Gamma^{1345}$ in place of $\Phi$ and index
$5$ in place of $2$, we get that $\ZZ(\Gamma^{134}-\Gamma^{345})\le
\QQ(\Gamma^{15}) = \QQ(\kappa^{15})$. Similarly,
$\ZZ(\Gamma^{345}-\Gamma^{234})\le \QQ(\kappa^{25})$, and so
\[
\QQ(\kappa^{12}) \le \QQ(\kappa^{15})+\QQ(\kappa^{25})+\delta.
\]
Since this holds for every $\delta>0$, we get that $\QQ(\kappa^{12})\le
\QQ(\kappa^{15}) + \QQ(\kappa^{25})$, proving that $\QQ$ is metrical.

Now $\QQ(\psi)\le \RR(\psi)<1$ but $\QQ(\sigma)\ge 1$, so the hypothesis of the
theorem is violated. This proves \eqref{EQ:HXY}.

\smallskip

This implies the (seemingly) stronger statement that
\begin{equation}\label{EQ:HXY1}
\Hf_1 \cap \Hf_2^\delta\cap \Hf_3^\delta\not=\emptyset
\end{equation}
for all $\delta>0$. Indeed, if $\Phi \in \Hf_1^{\delta/2} \cap
\Hf_2^{\delta/2}\cap \Hf_3^{\delta/2}$, then there is a measure $\Phi'\in
\Hf_1$ such that $d_\tv(\Phi,\Phi')<\delta/2$, and then $\Phi'\in \Hf_1 \cap
\Hf_2^\delta\cap \Hf_3^\delta$.

Our next step is to prove that for every $\delta>0$,
\begin{equation}\label{EQ:HXY2}
\Hf_1 \cap \Hf_2 \cap \Hf_3^\delta\not=\emptyset.
\end{equation}
Indeed, let $\Phi \in \Hf_1 \cap \Hf_2^{\delta/3}\cap \Hf_3^{\delta/3}$. By
$d_\tv(\Phi,\Hf_2)<\delta/3$ it follows that $\|\Phi_-\|<\delta/3$. Consider
the measure $\Psi=\Phi_++\Phi_-^* \in\ca_+(\AA^4)$, then
\begin{align*}
\Psi^{134}-\Psi^{234} &=(\Phi_+)^{134} + (\Phi_-^*)^{134} - (\Phi_+)^{234} - (\Phi_-^*)^{234}\\
&=(\Phi_+)^{134} + (\Phi_-)^{234} - (\Phi_+)^{234} - (\Phi_-)^{134}=\Phi^{134} -\Phi^{234}=\overline{\sigma}.
\end{align*}

Thus $\Psi\in \Hf_1\cap \Hf_2$. Furthermore,
\begin{equation}\label{EQ:DTVPH3}
d_\tv(\Psi,\Hf_3) \le d_\tv(\Phi,\Hf_3) + \|\Phi-\Psi\|< \frac13\delta + 2\|\Phi_-\| < \delta,
\end{equation}
so $\Psi\in \Hf_3^\delta$. The multicommodity flow $\Psi$ satisfies
\eqref{EQ:FLO-M-LOAD2} and \eqref{EQ:FLO-M-FEAS}, and it is easy to check that
it violates capacity $\psi$ by at most $\|\Psi^{12}\setminus\psi\| \le
d_\tv(\Psi,\Hf_3) <\delta$.

This completes the proof of Theorem \ref{THM:MULTI-COMM-M}.

\medskip

\noindent{\bf Acknowledgement.} My thanks are due to Mikl\'os Ab\'ert,
Alexander Kechris, Tam\'as Keleti, D\'avid Kunszenti-Kov\'acs and Mikl\'os
Laczkovich for their valuable help with this research. I am also indebted to
the anonymous referee for suggesting many improvements to the paper.

\end{document}